\documentclass[11pt,reqno,tbtags]{amsart}

\usepackage{amsthm,amssymb,amsfonts,amsmath}
\usepackage{amscd} 
\usepackage{mathrsfs}
\usepackage[numbers, sort&compress]{natbib} 
\usepackage{subfigure, graphicx}
\usepackage{vmargin}
\usepackage{setspace}
\usepackage{paralist}
\usepackage[pagebackref=false]{hyperref}
\usepackage{color}
\usepackage[normalem]{ulem}
\usepackage{stmaryrd} 
\usepackage[refpage,noprefix,intoc]{nomencl} 
\usepackage{setspace}
\usepackage{cancel}

\usepackage{tikz}
\usetikzlibrary{patterns}

\usepackage{sidecap}
\usepackage[font=small, labelfont=bf, labelsep=period]{caption}

\providecommand{\R}{}

\providecommand{\N}{}

\renewcommand{\R}{\mathbb{R}}

\renewcommand{\N}{{\mathbb N}}

\newcommand{\E}[1]{{\mathbf E}\left[#1\right]}										
\newcommand{\e}{{\mathbf E}}

\newcommand{\p}[1]{{\mathbf P}\left\{#1\right\}}

\newcommand{\I}[1]{{\mathbf 1}_{[#1]}}
\newcommand{\set}[1]{\left\{ #1 \right\}}
 
\newcommand{\probC}[2]{\mathbf{P}\set{#1 \; \left|  \; #2 \right. }}

\newcommand{\expC}[2]{\mathbf{E}\set{#1 \; \left|  \; #2 \right. }}

\newcommand\cT{{\mathcal T}}

\newcommand{\bN}{\mathbf{N}}

\newcommand{\eqdist}{\ensuremath{\stackrel{\mathrm{d}}{=}}}

\newcommand{\pran}[1]{\left(#1\right)}
\providecommand{\eps}{}
\renewcommand{\eps}{\epsilon}
\providecommand{\ora}[1]{}
\renewcommand{\ora}[1]{\overrightarrow{#1}}
\DeclareRobustCommand{\SkipTocEntry}[5]{} 

\newtheorem{thm}{Theorem}
\newtheorem{lem}[thm]{Lemma}
\newtheorem{prop}[thm]{Proposition}
\newtheorem{cor}[thm]{Corollary}

\newtheorem{conj}{Conjecture}

\makenomenclature										
\graphicspath{ {./images/} }
\usepackage{wrapfig}

\newcommand{\vset}{\ensuremath{v}}
\newcommand{\eset}{\ensuremath{e}}
\newcommand{\rt}{\ensuremath{r}}
\newcommand{\dseq}{\ensuremath{\mathrm{d}}}
\newcommand{\wid}{\ensuremath{\mathrm{wid}}}
\newcommand{\hyt}{\ensuremath{\mathrm{ht}}}
\providecommand{\deg}{}
\renewcommand{\deg}{\ensuremath{d}}
\renewcommand{\dseq}{\ensuremath{\mathrm{d}}}
\newcommand{\tset}{\ensuremath{\mathscr{T}}}
\newcommand{\dstat}{\ensuremath{\mathrm{n}}}

\newcommand{\spine}{\ensuremath{S}}
\newcommand{\bd}{\ensuremath{\mathbf{d}}}
\newcommand{\varmod}{\ensuremath{\mathrm{v}}}
\newcommand{\oexp}{\ensuremath{\mathrm{oe}}}
\newcommand{\wts}{\ensuremath{\mathrm{w}}}
\newcommand{\sgt}{\ensuremath{\mathcal{T}}}

\numberwithin{equation}{section}
\numberwithin{thm}{section}

\usepackage{framed}

\definecolor{ccel}{rgb}{0.36, 0.54, 0.66}

\begin{document}
\setstretch{1.2}
\title{Universal height and width bounds for random trees} 
\address{Department of Mathematics and Statistics, McGill University, Montr\'eal, Canada}
\author{Louigi Addario-Berry}
\email{louigi.addario@mcgill.ca}
\author{Anna Brandenberger}
\email{anna.brandenberger@mail.mcgill.ca}
\author{Jad Hamdan}
\email{jad.hamdan@mail.mcgill.ca}
\author{C\'eline Kerriou}
\email{celine.kerriou@mail.mcgill.ca}

\date{September 25, 2021; revised April 25, 2022} 


\begin{abstract} 

We prove non-asymptotic stretched exponential tail bounds on the height of a randomly sampled node in a random combinatorial tree, which we use to prove bounds on the heights and widths of random trees from a variety of models. Our results allow us to prove a conjecture and settle an open problem of Janson~\cite{MR2908619}, and nearly prove another conjecture and settle another open problem from the same work (up to a polylogarithmic factor). 

The key tool for our work is an equivalence in law between the degrees along the path to a random node in a random tree with given degree statistics, and a random truncation of a size-biased ordering of the degrees of such a tree. We also exploit a Poissonization trick introduced by Camarri and Pitman \cite{MR1741774} in the context of inhomogeneous continuum random trees, which we adapt to the setting of random trees with fixed degrees. 

Finally, we propose and justify a change to the conventions of branching process nomenclature: the name ``Galton--Watson trees'' should be permanently retired by the community, and replaced with the name ``Bienaym\'e trees''. 
\end{abstract}

\maketitle



\section{\bf Introduction}\label{sec:intro} 


This paper concerns the height and width of random plane trees, and applications of bounds thereof to the study of random simply generated trees and to the family trees of branching processes. Our results in particular allow us to settle two conjectures from \cite{MR2908619}, and to nearly settle two others. 

By a plane tree, we mean a finite rooted tree $t=(\vset(t),\eset(t))$ in which the set of children of each node is endowed with a left-to-right order. The root of $t$ is denoted $\rt(t)$. The {\em degree} of a node $v \in \vset(t)$, denoted $\deg_t(v)$, is its number of children in $t$, so leaves have degree $0$ and all other nodes have strictly positive degree. 

The {\em degree statistics} of $t$ is the sequence $\dstat_t=(\dstat_t(c),c \ge 0)$, where $\dstat_t(c) = |\{v \in \vset(t): \deg_t(v)=c\}|$ is the number of nodes of $t$ with $c$ children. 
Note that 
\[
|v(t)| 
= \sum_{c \ge 0} \dstat_t(c)
=1+\sum_{c \ge 0} c\dstat_t(c)
=1+|e(t)| \, .
\]
A sequence $\dstat=(\dstat(c),c \ge 0)$ is the degree statistics of some tree if and only if $\sum_{c \ge 0} \dstat(c)=1+\sum_{c \ge 0} c\dstat(c)$. For such sequences, we write $\tset_{\dstat}$ for the set of plane trees with degree statistics $\dstat$, and write
\[
\tset_{\dstat}^{\bullet}
= \{(t,v): t \in \tset_{\dstat}, v \in v(t)\}.
\]
A {\em marked tree} is a pair $(t,v)$ where $t$ is a plane tree and $v \in v(t)$; so the elements of $\tset_{\dstat}^{\bullet}$ are precisely the marked trees with degree statistics $\dstat$.

For a node $v \in \vset(t)$, the {\em height} of $v$, denoted $|v|$, is the graph distance from $v$ to $\rt(t)$. The height of $t$, denoted $\hyt(t)$, is $\max(|v|: v \in \vset(t))$. The {\em width} of $t$ at level $k$, denoted $\wid(t,k)$, is $|\{v \in t: |v|=k\}|$, and the width of $t$, denoted $\wid(t)$, is $\max(\wid(t,k),k \ge 0)$. 
 
Given a a sequence $\dstat=(\dstat(c),c \ge 0)$ of non-negative real numbers, for $p>0$, we write $|\dstat|_p = (\sum_{c \ge 0} c^p \dstat(c))^{1/p}$. Note that for a plane tree $t$, we have $|\dstat_t|_1+1 = |t|$. 

We prove the following non-asymptotic tail bounds on the height of a randomly sampled node in a random plane tree with given degree statistics. For a finite set $S$ we write $X \in_u S$ to mean that $X$ is a uniformly random element of the set $S$. 
\begin{thm}
\label{thm:main}
Fix degree statistics $\dstat=(\dstat(c),c \ge 0)$ and let $(T,V) \in_u \tset_\dstat^{\bullet}$. Then for all $\beta > 17^{3/2}$, 
\[
\p{|V| > \beta \frac{|\dstat|_1}{(|\dstat|_2^2-\dstat(1))^{1/2}}}
\le 
\exp\pran{-\frac{\beta^{1/3}}{3}\frac{|\dstat|_1}{(|\dstat|_2^2-\dstat(1))^{1/2}}} + 2\exp\pran{-\frac{\beta^{2/3}}{24}}\, ,
\]
and if $\dstat(1)=0$ then for all $\ell \ge 1$, 
\[
\p{|V| \ge \ell}
\le 
\exp\pran{-\frac{\ell^2}{2|\dstat|_1}}\, .
\]
\end{thm}

A related bound, recently proved by Marzouk \cite[Proposition 5]{marzouk}, strengthens the first of the two bounds stated in the preceding theorem, up to constant factors.
These bounds, interesting in their own right, also have several consequences for the family trees of branching processes, which are summarized in our other main theorems,  below. In order to state the theorems, we need a little more terminology.

Given a tree $t$, write $t^{\le k}$ for the subtree of $t$ consisting of all nodes $u \in \vset(t)$ with $|u| \le k$. 

Let $\mu$ be a probability distribution with support $\N$ (by which we mean that $\mu(\N)=1$). By a {\em Bienaym\'e tree with offspring distribution $\mu$}, we mean the family tree $T$ of a branching process with offspring distribution $\mu$.\footnote{We propose this name as an alternative to the term ``Galton--Watson tree'' since (a) Bienaym\'e's introduction of such trees both predates that of Galton and Watson and is more mathematically correct - see \cite{bienayme} and \cite[Pages 83-86]{cournot}; and (b) we prefer not to honour the founder of eugenics, Francis Galton, by continuing to attach his name to these mathematical objects. (Galton wrote: ``We greatly want a brief word to express
the science of improving stock, which is by no means confined to questions of judicious
mating, but which, especially in the case of man, takes cognisance of all influences that
tend in however remote a degree to give to the more suitable races or strains of blood a
better chance of prevailing speedily over the less suitable than they otherwise would have
had. The word eugenics would sufficiently express the idea'' \cite[Page 25]{galton83}.)} 
The law of $T$ is uniquely determined by the property that for any plane tree $t$ of height at most $k$, 
\[
\p{T^{\le k}=t} = \prod_{v \in t^{\le k-1}} \mu(d_t(v))\, .
\]
In the preceding formula and below, we write $\mu(k) = \mu(\{k\})$ for readability. 

For $n \in \N$, if $\p{|\vset(T)|=n}>0$ then we define a Bienaym\'e tree conditioned to have size $n$ in the natural way: this is a random tree $T_n$ such that for any plane tree $t$ with $n$ vertices,
\[
\p{T_n=t} = \probC{T=t}{|\vset(T)|=n}\, .
\]

Finally, for a measure $\mu$ on $\R$, for $p > 0$ we write $|\mu|_p := (\int_\R |x|^p\mu(\mathrm{d}x))^{1/p}$. This agrees with the above notation $|\dstat|_p$ for sequences $\dstat=(\dstat(c),c \ge 0)$, by interpreting the sequence as the discrete measure assigning mass $\dstat(c)$ to each non-negative integer $c$. 
\begin{thm}\label{thm:main2}
Fix a probability distribution $\mu$ with support $\N$, with $|\mu|_1 \le 1$ and $|\mu|_2 = \infty$. For $n \in \N$, let $T_n$ be a Bienaym\'e tree with offspring distribution $\mu$, conditioned to have size $n$, and let $V_n$ be a uniformly random node of $T_n$. Then $\wid(T_n)/n^{1/2} \to \infty$, $|V_n|/n^{1/2} \to 0$, and $\hyt(T_n)/(n^{1/2}\log^3 n) \to 0$. All convergence results hold both in probability and in expectation, as $n \to \infty$. 
\end{thm}
\begin{thm}\label{thm:main3}
Fix a probability distribution $\mu$ with support $\N$, with $|\mu|_1 <1$ and with $\sum_{c \ge 0} e^{tc}\mu(c)=\infty$ for all $t$. For $n \in \N$ let $T_n$ be a Bienaym\'e tree with offspring distribution $\mu$, conditioned to have size $n$, and let $V_n$ be a uniformly random node of $T_n$. Then $\wid(T_n)/n^{1/2} \to \infty$, $|V_n|/n^{1/2} \to 0$, and $\hyt(T_n)/(n^{1/2}\log^3 n) \to 0$. All convergence results hold both in probability and in expectation, as $n \to \infty$. 
\end{thm}
The results of Theorems~\ref{thm:main2} and~\ref{thm:main3} are close analogues of Conjectures 21.5 and 21.6 and Problems 21.7 and 21.8 from~\cite{MR2908619}, but those conjectures are stated for the slightly more general model of {\em simply generated trees}. In Section~\ref{sec:simply_generated} we define simply generated trees, state the aforementioned conjectures and problems precisely, and explain how to use Theorem~\ref{thm:main} to prove Conjecture~21.6 and solve Problem~21.8 from \cite{MR2908619}, and to nearly prove Conjecture 21.5 and nearly solve Problem~21.7 from the same paper.\footnote{``Nearly prove'' and ``nearly solve'' rather than ``prove'' and ``solve'' due to the presence of a $\log^{3} n$ factor in two of our bounds.} The key fact about simply generated trees is that, like conditioned Bienaym\'e trees, they are uniformly random conditional on their degree statistics, which allows us to apply Theorem~\ref{thm:main} to them. 

We also prove height and width bounds for conditioned Bienaym\'e trees which hold without any assumptions on the offspring distribution at all, aside from the requirement that the resulting family trees have both leaves and branch points. \begin{thm}\label{thm:main4}
There exists a constant $C > 0$ such that the following holds. Fix a probability distribution $\mu$ with support $\N$ with $\mu(0)+\mu(1)<1$. For $n \in \N$, let $T_n$ be a Bienaym\'e tree with offspring distribution $\mu$, conditioned to have size $n$, and let $V_n$ be a uniformly random node of $T_n$. 
Then $\e |V_n| \le Cn^{1/2}/(1-\mu(0)-\mu(1))^{1/2}$, 
\[
\E{
\wid(T_n)} \ge \frac{(1-\mu(0)-\mu(1))^{1/2}}{C} n^{1/2} \quad \mbox{ and }\quad
\E{\hyt(T_n)}\le C\frac{n^{1/2} \log^{3} n}{(1-\mu(0)-\mu(1))^{1/2}}. 
\]
\end{thm}

\subsection{Discussion}
There is a substantial amount of past work on the heights and widths of random Bienaym\'e trees and random combinatorial trees \cite{MR2956056,MR3077536,MR3916103,MR3651047,MR3335012}, and bounds on these quantities, particularly the height, often feature in scaling limit theorems for random trees and associated objects \cite{ss21,MR4132643,MR3947331,MR3551197,MR3188597}. The works \cite{MR2956056,MR3077536,MR3916103} all bound the height via the study of the {\em depth-first exploration process} of the tree. This technique gives bounds which are frequently tight, up to constant factors, for trees whose offspring distributions are sufficiently light tailed (e.g. with finite variance). However, it does not appear well-suited to studying trees with heavy-tailed degrees (in which case the depth-first queue length is a poor proxy for the height).

For critical conditioned Bienaym\'e trees with finite variance ($|\mu|_1=1$, $|\mu|_2<\infty$), sub-Gaussian tail bounds for $n^{-1/2}\hyt(T_n)$ and $n^{-1/2}\wid(T_n)$ are known \cite{MR3077536}. However, the authors of that paper state that they ``are not aware of any results [for the height and width] that hold for arbitrary offspring distributions.'' As far as the authors of the current paper are aware, this is still the case, and this paper is the first work to provide such results. 

We do not expect that the stretched exponential tail bound of Theorem~\ref{thm:main} is tight. However, it is not completely clear what form an optimal bound ought to take. We now record some observations which limit how quickly the optimal bounds can decay, to help provide a sense of the potential complexities. These observations in particular show that the exponents $1/3$ and $2/3$ in Theorem~\ref{thm:main} can not be replaced by any values strictly greater than $1$, which means that one can not hope for sub-Gaussian tail bounds like those we prove for degree statistics with $\dstat(1)=0$ to hold in general.  (The computations underlying the observations are not fully spelled out here but are not too complicated.)

First, fix $\alpha \in (1,2)$, and suppose that $|\mu|_1=1$ and $\mu(k,\infty) =(1+o(1))c k^{-\alpha}$ as $k \to \infty$, so $\mu$ is a critical offspring distribution in the domain of attraction of an $\alpha$-stable law.
In this setting, it is known \cite[Theorem 1.5]{MR3634265} that 
\begin{equation}\label{eq:alpha_stable_bd}
\p{\hyt(T_n) \ge c n^{1-1/\alpha}} \asymp c^{1+\alpha/2}\exp(- (\alpha-1)^{1/(\alpha-1)}c^{\alpha})
\end{equation}
as first $n\to\infty$, then $c \to \infty$. 
Such a tree $T_n$ will typically have $\Theta(n k^{-\alpha-1})$ nodes of degree $k$ for $k \le n^{1/\alpha}$ and no nodes of degree much larger than $n^{1/\alpha}$, and so will satisfy $|\dstat_{T_n}|_2^2 \asymp n^{2/\alpha}$. Since $\alpha$ can be taken arbitrarily close to $1$, comparing the upper bound from Theorem~\ref{thm:main} with (\ref{eq:alpha_stable_bd}) shows that the exponent $2/3$ in Theorem~\ref{thm:main} can not be replaced with anything strictly greater than one. 

Second, consider degree statistics of the form $\dstat=(k,k,\ldots,0,1,0,\ldots)$, corresponding to a tree with $n=2k+1$ nodes, with a single node of degree $k$, and $k$ nodes each of degrees $0$ and $1$. For such degree statistics, $|\dstat|_1/(|\dstat|_2^2-\dstat(1))^{1/2} = \Theta(1)$. 
Moreover, it is not hard to see that with high probability a random tree with these degree statistics has height $\Theta(\log n)$, so there is $\delta >0$ such that the probability that a randomly sampled node has height at least $\delta\log n$ is at least $(\delta \log n)/n$. Combining these observations shows that neither the exponent $1/3$ nor the exponent $2/3$ in the first bound in Theorem~\ref{thm:main} can in general be replaced with anything greater than $1$. (Marzouk's result \cite[Proposition 5]{marzouk} shows that one {\em can} essentially replace both constants $1/3$ and $2/3$ by $1$; the above observations show that this is then best possible.)

We conclude the discussion with a word about Theorem~\ref{thm:main4}. 
The dependence on $\mu(0)$ and $\mu(1)$ in  that theorem is necessary; if $\mu(0)+\mu(1)=1$ then with probability one $T_n$ is a path with $n$ vertices, which has width $1$ and height $n-1$. Moreover, the form of the dependence in the theorem is essentially optimal. To see this, suppose that $\mu(1)=1-\eps$ and $\mu(0)=\mu(2)=\eps/2$. 
Then with high probability $T_n$ will have 
$(1+o(1))(1-\eps)n$ vertices with exactly one child. Let $\hat{T}_n$ be the tree obtained from $T_n$ by suppressing all vertices with exactly one child, so that $\hat{T}_n$ has only nodes with $0$ or $2$ children, and $T_n$ can be recovered from $\hat{T}_n$ by subdividing edges. Then $\hat{T}_n$ has size $(1+o(1))\eps n$ with high probability, and is a uniform binary tree conditional on its size, so has height $\Theta((\eps n)^{1/2})$ and width $\Theta((\eps n)^{1/2})$ in probability. Each edge of $\hat{T}_n$ is subdivided $\Theta(\eps^{-1})$ times on average in $T_n$, from which it is easy to believe (and not too hard to prove) that $T_n$ has height $\Theta((n/\eps)^{1/2})=\Theta((n/(1-\mu(0)-\mu(1)))^{1/2})$ and width $\Theta((\eps n)^{1/2}) = \Theta(((1-\mu(0)-\mu(1))n)^{1/2})$ in probability and in expectation.

\subsection{Notation}
For a sequence $(r_n,n \ge 1)$ of real numbers, we write $r_n=\oexp(1)$ if there exists $c > 0$ such that $r_n \le e^{-cn}$ for all $n$ sufficiently large. 
Given a sequence of events $(E_n,n \ge 1)$, 
we say that $E_n$ occurs {\em with very low probability} (and that $E_n^c$ occurs {\em with very high probability}) if $\p{E_n}= \oexp(1)$.

\section{\bf An overview of the proofs.}

\subsection{\bf A sampler for the height of the marked node}
Fix degree statistics $\dstat$, and let $(T,V) \in_u \tset_{\dstat}^{\bullet}$. The tool which unlocks all the results of the paper is a sampling procedure which generates a random variable with the same law as $|V|$.  To describe the sampling procedure, some notation is needed. 

Given degree statistics $\dstat=(\dstat(c),c \ge 0)$, we say a random vector $D=(D_1,\ldots,D_n)$ is a {\em size-biasing of $\dstat$} if $n =\sum_{c \ge 0} \dstat(c) $ and for any sequence $\dseq=(d_1,\ldots,d_n)$ such that $|\{i \in [n]: d_i=c\}|=\dstat(c)$ for all $c$, the following holds. For each $1 \le k\le n$, 
\begin{align}\label{eq:size_biasing}
	\probC{D_k=d_k}{(D_1,\ldots,D_{k-1})=(d_1,\ldots,d_{k-1})} & = \frac{d_k(\dstat(d_k) - w((d_1,...,d_{k-1}),d_k))}{|\dstat|_1-d_1-\ldots-d_{k-1}}\, , 
\end{align}
where $w((d_1,...,d_{k-1}), d) = |\{i\in [k-1]:d_i = d\}|$ for $d\geq 0$. The fraction is interpreted as equal to $1$ if $d_1+\ldots+d_{k-1}=|\dstat|_1$. The definition implies that $|\{1 \le k \le n: D_k=c\}|=\dstat(c)$ almost surely, for all $c \ge 0$. It also implies that the final $\dstat(0)\ge 1$ entries of $D$ all equal $0$, and in particular $D_n=0$. 

\begin{prop}\label{prop:sampler}
Fix degree statistics $\dstat$ and write $n=|\dstat|_1+1$. Let $(T,V) \in_u \tset_{\dstat}^{\bullet}$, and let $D=(D_1,\ldots,D_n)$ be a size-biasing of $\dstat$. 
Next let $(U_1,\ldots,U_n)$ be independent Uniform$[0,1]$ random variables, independent of $D$. For $i \in [n]$ let 
\[
A_i = 
\begin{cases}
1	& \mbox{ if }U_i \le\frac{1+\sum_{j=1}^{i-1} (D_j-1)}{n+1-i}\\
0	& \mbox{ otherwise}. 
\end{cases}
\]
Finally, let $M = \min(i: A_i = 1)$. Then $|V|\eqdist M-1$. 
\end{prop}
In the above proposition, note that since $D_n=0$, when $i=n$ we have 
\[
\frac{1+ \sum_{j=1}^{i-1} (D_j-1)}{n+1-i}
= 1+\sum_{j=1}^{n-1} (D_j-1) = 1,
\]
so $A_n=1$ and thus $M \le n$. 

Theorem~\ref{thm:main} is an essentially immediate consequence of Proposition~\ref{prop:sampler} together with the next result. 
\begin{thm}
\label{thm:main_mod}
Fix degree statistics $\dstat$, write $n=|\dstat|_1 + 1$, and let $D=(D_1,\ldots,D_n)$ be a size-biasing of $\dstat$. 
Next let $(U_1,\ldots,U_n)$ be independent Uniform$[0,1]$ random variables, independent of $D$. For $i \in [n-1]$ let 
\[
B_i = 
\begin{cases}
1	& \mbox{ if }U_i \le\frac{\sum_{j=1}^{i-1} (D_j-1)}{n-i}\\
0	& \mbox{ otherwise}. 
\end{cases}
\]
Finally, let $\sigma = \inf(i: B_i = 1)$. If 
$\dstat(1)>0$ then for all $\beta > 17^{3/2}$, 
\[
\p{\sigma > \beta \frac{|\dstat|_1}{(|\dstat|_2^2-\dstat(1))^{1/2}}}
\le
\exp\pran{-\frac{\beta^{1/3}}{3}\frac{|\dstat|_1}{(|\dstat|_2^2-\dstat(1))^{1/2}}} + 2\exp\pran{-\frac{\beta^{2/3}}{24}}\, ,
\]
and if $\dstat(1)=0$ then for all $\ell \ge 1$, 
\[
\p{\sigma \ge \ell} \le e^{-(\ell-1)^2/(2|\dstat|_1)}\, .
\]

\end{thm}
Note that in Theorem~\ref{thm:main_mod}, we have $B_i=1$ iff $U_i \le \tfrac{\sum_{j=1}^{i-1} (D_j-1)}{n-i}$, whereas in Proposition~\ref{prop:sampler}, we have $A_i=1$ iff $U_i \le \tfrac{1+\sum_{j=1}^{i-1} (D_j-1)}{n+1-i}$. 
Since $\tfrac{a+1}{b+1} \ge \tfrac{a}{b}$ whenever $\tfrac{a}{b} < 1$, it follows that the random variable $\sigma=\inf(i: B_i=1)$ stochastically dominates the random variable $M=\min(i:A_i=1)$. Thus, upper tail bounds for $\sigma$ automatically apply to $M$. Since $|V|\eqdist M-1$, in view of this observation, the bounds of Theorem~\ref{thm:main} follow immediately from those of Theorem~\ref{thm:main_mod}.  

\subsection{\bf Moving from random marked trees to Bienaym\'e trees}

Once Theorem~\ref{thm:main} is established, the primary work in proving the other results of the paper is to understand the degree statistics of conditioned Bienaym\'e trees, under various assumptions on their offspring distributions. We prove the following bounds. 

\begin{prop}\label{prop:dstats_infinite_variance}
Fix a probability distribution $\mu$ with support $\N$ with $|\mu|_1 \le 1$ and $|\mu|_2 = \infty$. For $n \in \N$ let $T_n$ be a Bienaym\'e tree with offspring distribution $\mu$, conditioned to have size $n$, and let $\dstat_{T_n}$ be the degree statistics of $T_n$. Then for any $C > 0$, with very high probability $|\dstat_{T_n}|_2^2 \ge C |\dstat_{T_n}|_1$. 
\end{prop}
\begin{prop}\label{prop:dstats_zero_radius}
Fix a probability distribution $\mu$ with support $\N$, with $|\mu|_1 < 1$ and with $\sum_{c \ge 0} e^{tc}\mu(c)=\infty$ for all $t>0$. For $n \in \N$ let $T_n$ be a Bienaym\'e tree with offspring distribution $\mu$, conditioned to have size $n$, and let $\dstat_{T_n}$ be the degree statistics of $T_n$. Then for any $C > 0$, with very high probability $|\dstat_{T_n}|_2^2 \ge C |\dstat_{T_n}|_1 $. 
\end{prop}
\begin{prop}\label{prop:dstats_finite_variance}
Fix a probability distribution $\mu$ with support $\N$ and with $\mu(0) +\mu(1) < 1$. For $n \in \N$ let $T_n$ be a Bienaym\'e tree with offspring distribution $\mu$, conditioned to have size $n$, and let $\dstat_{T_n}$ be the degree statistics of $T_n$. Then for any $\eps > 0$, with very high probability $|\dstat_{T_n}|_2^2-\dstat_{T_n}(1)\ge |\dstat_{T_n}|_1\cdot 4(1-\mu(0)-\mu(1)-\eps)$. 
\end{prop}
In the remainder of this section, we explain how Theorems~\ref{thm:main2},~\ref{thm:main3} and~\ref{thm:main4} follow from Propositions~\ref{prop:dstats_infinite_variance},~\ref{prop:dstats_zero_radius} and~\ref{prop:dstats_finite_variance} together with the bound from Theorem~\ref{thm:main}.
\begin{proof}[Proof of Theorem~\ref{thm:main2}]
For $n \ge 1$, let $V_n$ be a uniformly random node of $T_n$. 
Next, fix $\eps > 0$ small, and let $C=C(\eps)=1+\eps^{-4}$. Then let $E_n$ be the event that 
$|\dstat_{T_n}|_2^2 \ge C|\dstat_{T_n}|_1$. 

Now fix any degree statistics $\dstat$ with $\p{\dstat_{T_n}=\dstat} > 0$. Conditionally given that $\dstat_{T_n}=\dstat$, then $(T_n,V_n) \in_u \tset_\dstat^{\bullet}$. Thus, if $|\dstat|_2^2 \ge C|\dstat|_1$, then $|\dstat|_2^2-\dstat(1) \ge (C-1)|\dstat|_1=|\dstat|_1/\eps^4$, 
so for all $\beta \ge 17^{3/2}$ we have 
\begin{align}
\probC{|V_n| \ge \beta\eps^2 |\dstat|_1^{1/2}}{\dstat_{T_n}=\dstat}
& \le 
\probC{|V_n| \ge \beta\frac{|\dstat|_1}{(|\dstat|_2^2-\dstat(1))^{1/2}}}{\dstat_{T_n}=\dstat}
\nonumber\\
& \le \exp\pran{-\frac{\beta^{1/3}}{3}} + 2\exp\pran{-\frac{\beta^{2/3}}{24}}\, ,\label{eq:vnht_cond_bd}
\end{align}
where for the second equality we have used the bound from Theorem~\ref{thm:main} together with the fact that $|\dstat|_2 \le |\dstat|_1$ so $\tfrac{|\dstat|_1}{(|\dstat|_2^2-\dstat(1))^{1/2}} \ge 1$. 

We now use the bound 
\begin{align}\label{eq:vht_expbd}
\E{|V_n| \I{|V_n| \ge \beta\eps^2|\dstat|_1^{1/2}}}
& \le n\p{E_n^c} + 
\sup_{\dstat:|\dstat|_2^2 \ge C|\dstat|_1} \expC{|V_n| \I{|V_n| \ge \beta\eps^2 |\dstat|_1^{1/2}}}{\dstat_{T_n}=\dstat}\, .
\end{align}
The first term is $o(1)$ since $E_n$ occurs with very high probability by Proposition~\ref{prop:dstats_infinite_variance}. To bound the second term we use that for any non-negative integer random variable $X$ and any $z > 0$, we have 
\[
\E{X\I{X \ge \beta z}} \le \lfloor \beta z\rfloor + z\int_{\beta}^\infty \p{X \ge xz}\mathrm{d}x\,
\le  z\pran{\beta + \int_\beta^\infty \p{X \ge xz}\mathrm{d}x}\, .
\]
Using the bound from \eqref{eq:vnht_cond_bd}, it follows that 
\begin{align*}
& \sup_{\dstat:|\dstat|_2^2 \ge C|\dstat|_1} \expC{|V_n| \I{|V_n| \ge \beta\eps^2 |\dstat|_1^{1/2}}} {\dstat_{T_n}=\dstat}\\
& \le \eps^2|\dstat|_1^{1/2} \pran{\beta + 
 \int_\beta^\infty
\exp\pran{-\frac{x^{1/3}}{3}}\mathrm{d}x
 + \int_\beta^\infty 2\exp\pran{-\frac{x^{2/3}}{24}}\mathrm{d}x\,} .
\end{align*}
Taking $\beta=1/\eps$, for $\eps$ small the first integral is $O(e^{-\eps^{-1/3}/3})$ and the second is $O(e^{-\eps^{-2/3}/24})$, so \eqref{eq:vht_expbd} gives that 
\[
\E{|V_n| \I{|V_n| \ge \eps|\dstat|_1^{1/2}}}
\le \eps|\dstat|_1^{1/2} + O(\eps^2 e^{-\eps^{-1/3}/3}|\dstat|_1^{1/2})+ O(\eps^2 e^{-\eps^{-2/3}/24}|\dstat|_1^{1/2})\,.
\]
Since $\eps^2 e^{-\eps^{-1/3}/3}=O(\eps)$ and $\eps^2 e^{-\eps^{-2/3}/24} = O(\eps)$ for $\eps > 0$ small, this implies that 
\[
\e{|V_n|} = O(\eps |\dstat|_1^{1/2}), 
\]
and so $\e|V_n| = o(|\dstat|_1^{1/2})$ as $\eps>0$ can be chosen arbitrarily small. Since $|\dstat|_1 = (n-1)$, it follows that $n^{-1/2}|V_n| \to 0$ in probability and in expectation. 


Next, for any $\eps > 0$, with $C=1+\eps^{-4}$ as above, taking $\beta = (6\log n)^{3}$ in \eqref{eq:vnht_cond_bd}, we obtain that 
\begin{align*}
& \p{|V_n| \ge (6 \log n)^3\eps^2(n-1)^{1/2}}\\
&\le \oexp(1) + \sup_{\dstat:|\dstat|_2^2 \ge C|\dstat|_1}
\probC{|V_n| \ge (6 \log n)^3\eps^2(n-1)^{1/2}}{\dstat_{T_n}=\dstat}\\
& \le \oexp(1) + \exp\pran{-2 \log n} + 2\exp\pran{-\frac{3 \log^{2} n}{2}}\,\\
& = O\pran{\frac{1}{n^2}}\, .
\end{align*}
On the other hand, since $V_n$ is a uniformly random node of $T_n$, for any positive integer $h$ we have $\probC{|V_n|\ge h}{\hyt(T_n)\ge h} \ge 1/n$, so $\p{\hyt(T_n)\ge h}\le n\p{|V_n| \ge h}$ and thus 
\[
\p{\hyt(T_n) \ge (6 \log n)^3\eps^2(n-1)^{1/2}}=O\pran{\frac{1}{n}}.
\]
Since this holds for any $\eps > 0$, it follows that $\hyt(T_n)/((n-1)^{1/2}\log^3 n) \to 0$ in probability, and since 
\begin{align*}
\E{\hyt{ ( T_n ) }} & \le (6 \log n)^3\eps^2(n-1)^{1/2} + 
n\p{\hyt{ ( T_n ) } \ge (6 \log n)^3\eps^2(n-1)^{1/2}} \\
& \le (6 \log n)^3\eps^2(n-1)^{1/2}+ O(1)\, ,
\end{align*}
again for any $\eps > 0$, also $\hyt(T_n)/( (n-1)^{1/2}\log^3 n) \to 0$ in expectation. 

Finally, fix $\eps > 0$. For all $n$ large enough that $\e|V_n| \le \eps^2 n^{1/2}$, by Markov's inequality, 
\[
\p{|V_{n}| \ge \eps n^{1/2}} \le \eps.
\]
On the other hand, since $V_n$ is a uniformly random node of $T_n$, 
\[
\p{|V_{n}| \ge \eps n^{1/2}}
\ge 
\frac{1}{2} \p{ |\{u \in T_n: |u| \ge \eps n^{1/2}\}| \ge \frac{n}{2}},
\]
so 
\[
\p{ |\{u \in T_n: |u| \ge \eps n^{1/2}\}| \ge \frac{n}{2}} \le 2\eps. 
\]
Further, if $\{u \in T_n: |u| \ge \eps n^{1/2}\}| <n/2$ then there are more than $n/2$ nodes in the first $\eps n^{1/2}$ levels of the tree, so $\wid(T_n) \ge n^{1/2}/(2\eps)$. It follows that 
\[
\p{\wid(T_n) \ge \frac{n^{1/2}}{2\eps}} \ge 1-2\eps\, ;
\]
since $\eps > 0$ was arbitrary, it follows that $n^{-1/2}\wid(T_n) \to \infty$ in probability and in expectation.
\end{proof}
Theorem~\ref{thm:main3} follows from Proposition~\ref{prop:dstats_zero_radius} in exactly the same way as Theorem~\ref{thm:main2} follows from Proposition~\ref{prop:dstats_infinite_variance}, so we omit the details. The proof of Theorem~\ref{thm:main4} from Proposition~\ref{prop:dstats_finite_variance} is quite similar but not identical, so we do provide a (somewhat terser) explanation.
\begin{proof}[Proof of Theorem~\ref{thm:main4}]
Fix any degree statistics $\dstat$ and let $(T,V) \in_u \tset_{\dstat}^{\bullet}$. Then integrating the tail bound from Theorem~\ref{thm:main} over $\beta \ge  17^{3/2}$, in essentially the same way as in the proof of Theorem~\ref{thm:main2}, it follows that 
\begin{equation}\label{eq:universal_sample_bd}
\e|V| \le C \frac{|\dstat|_1}{(|\dstat|_2^2 - \dstat(1))^{1/2}}\, ,
\end{equation}
where $C>0$ is a universal constant.

Since $(|\dstat|_2^2-\dstat(1))^{1/2} \le |\dstat|_2 \le |\dstat|_1=n-1$, using the tail bound from Theorem~\ref{thm:main} with $\beta = (6 \log n)^3$, 
we also obtain that 
\[
\p{|V| \ge (6 \log n)^3 \frac{|\dstat|_1}{(|\dstat|_2^2 - \dstat(1))^{1/2}}}
\le \exp(-2\log n) + 2\exp\pran{- \frac{3\log^{2} n}{2}}\le \frac{1}{n},
\]
the last bound holding whenever $|\dstat|_1$ is sufficiently large. Since 
\[
\p{\hyt(T)\ge h}\le n \p{|V| \ge h}
\]
and $\hyt(T)\le n-1$, it follows that 
\begin{equation}\label{eq:uniform_ht_bd}
\E{\hyt(T)} \le C\log^{3}n \frac{|\dstat|_1}{(|\dstat|_2^2 - \dstat(1))^{1/2}}\, ,
\end{equation}
with $C>0$ again a universal constant. 

Now let $T_n$ be a Bienaym\'e tree with offspring distribution $\mu$ as in the statement of Theorem~\ref{thm:main4}, and let $V_n$ be a uniformly random node of $T_n$. 
Let $\eps = (1-\mu(0)-\mu(1))/2$ and let $E_n$ be the event that $|\dstat_{T_n}|_2^2-\dstat_{T_n}(1)\ge |\dstat_{T_n}|_1\cdot 4(1-\mu(0)-\mu(1)-\eps) = 2|\dstat_{T_n}|_1(1-\mu(0)-\mu(1))$. By Proposition~\ref{prop:dstats_finite_variance}, $E_n$ occurs with very high probability. 

On the event $E_n$ we have 
\[
\frac{|\dstat_{T_n}|_1}{(|\dstat_{T_n}|_2^2 - \dstat_{T_n}(1))^{1/2}}
\le \frac{n^{1/2}}{2^{1/2}(1-\mu(0)-\mu(1))^{1/2}}\, ,
\]
so by (\ref{eq:universal_sample_bd}) we have 
\begin{align*}
\e{|V_n|} & \le n\p{E_n^c}+ \mathop{\sup_{\dstat:|\dstat|_1=n +1}}_{|\dstat_{T_n}|_2^2-\dstat_{T_n}(1)\ge2|\dstat_{T_n}|_1(1-\mu(0)-\mu(1))} 
\expC{|V_n|}{\dstat_{T_n}=\dstat}\\
& \le o(1)+C\frac{n^{1/2}}{2^{1/2}(1-\mu(0)-\mu(1))^{1/2}}\, ,
\end{align*}
where in the second inequality we have used that $\p{E_n^c}=\oexp(1)$. The lower bound on $\E{\wid(T_n)}$ follows from this upper bound on $\e|V_n|$ just as in the proof of Theorem~\ref{thm:main2}. Finally, 
\begin{align*}
\E{\hyt(T_n)} & \le n\p{E_n^c} + \mathop{\sup_{\dstat:|\dstat|_1=n+1}}_{|\dstat_{T_n}|_2^2-\dstat_{T_n}(1)\ge2|\dstat_{T_n}|_1(1-\mu(0)-\mu(1))}
\expC{\hyt(T_n)}{\dstat_{T_n}=\dstat} \\
& \le o(1) + C\log^3 n \frac{n^{1/2}}{2^{1/2}(1-\mu(0)-\mu(1))^{1/2}}\, ,
\end{align*}
the second bound holding by (\ref{eq:uniform_ht_bd}) and since $\p{E_n^c}=\oexp(1)$. This establishes the requisite bound on $\E{\hyt(T_n)}$, and completes the proof. 
\end{proof}
\section{\bf Proof of Proposition~\ref{prop:sampler}}

We begin with some combinatorial definitions and facts which we will require for the proof. A {\em forest} is an ordered sequence $f=(t_1,\ldots,t_a)$ of plane trees. 
The degree statistics of $f$ is the sequence $\dstat_f=(\dstat_f(c),c \ge 0)$ where $\dstat_f(c)$ is the number of nodes of $f$ with $c$ children. 

Fix integers $1 \le a \le n$ and let $\dstat=(\dstat(c),c\ge 0)$ be a sequence of non-negative integers with $\sum_{c \ge 0} \dstat(c) = n$ and $\sum_{c \ge 0} c\dstat(c)=n-a$. Any forest with degree statistics $\dstat$ has $n$ nodes and is composed of $a$ trees. 
We write $\tset_{\dstat}$ to denote the set of forests with degree statistics $\dstat$. 
A single tree $t$ can be interpreted as a forest $f=(t)$, which makes this notation agree with and extend the previously introduced notation $\tset_\dstat$ for the set of trees with given degree statistics (in which case $a=1$, i.e., $|\dstat|_1 = n-1$).
By \cite[Exercise 6.2.1]{MR2245368}, it holds that 
\begin{equation}\label{eq:dstat_count}
|\tset_{\dstat}| = \frac{a}{n} {n \choose \dstat(c),c \ge 0} = \frac{a}{n} \frac{n!}{\prod_{c \ge 0} \dstat(c)!} \, .
\end{equation}

Next write $\tset_\dstat^\bullet$ for the set of forests with degree statistics $\dstat$  with a marked node, and $\tset_\dstat^{(1)}$ for the subset of $\tset_{\dstat}^\bullet$ where the mark is in the first tree:
\begin{align*}
\tset_\dstat^{\bullet} & = \{(f,v):f=(t_1,\ldots,t_a) \in \tset_\dstat,v \mbox{ is a node of }f\},\\
\tset_\dstat^{(1)} & = \{(f,v):f=(t_1,\ldots,t_a) \in \tset_\dstat,v \in t_1\}\, .
\end{align*}
For any forest $f \in \tset_\dstat$ there are $n$ ways to choose a node to mark, so $|\tset_\dstat^{\bullet}|=n|\tset_\dstat|$. Moreover, there is a natural $a$-to-$1$ correspondence between $\tset_\dstat^{\bullet}$ and $\tset_\dstat^{(1)}$: if $(f,v) \in \tset_\dstat^{\bullet}$ with $f=(t_1,\ldots,t_a)$ and $v \in t_i$, then 
\[
((t_i,t_{i+1},\ldots,t_a,t_1,\ldots,t_{i-1}),v) \in \tset_\dstat^{(1)}. 
\]
It follows that 
\begin{equation}\label{eq:marked_forest_count}
|\tset_\dstat^{(1)}| = \frac{n}{a}|\tset_{\dstat}| = {n \choose \dstat(c),c \ge 0} \, .
\end{equation}

\begin{figure}[htb]
    \color{black}  
    \centering

\tikzset{every picture/.style={line width=0.75pt}} 

\begin{tikzpicture}[x=0.75pt,y=0.75pt,yscale=-1,xscale=1]

\draw [color={rgb, 255:red, 0; green, 0; blue, 0 }  ,draw opacity=1 ]   (332.48,96.71) -- (332.48,123.99) ;
\draw   (352.93,42.16) -- (363.15,69.43) -- (342.7,69.43) -- cycle ;
\draw [color={rgb, 255:red, 208; green, 2; blue, 27 }  ,draw opacity=1, line width = 0.6mm ]   (332.61,15.17) -- (352.93,42.16) ;
\draw   (312.02,42.16) -- (322.25,69.43) -- (301.8,69.43) -- cycle ;
\draw [color={rgb, 255:red, 208; green, 2; blue, 27 }  ,draw opacity=1, line width = 0.6mm ][fill={rgb, 255:red, 208; green, 2; blue, 27 }  ,fill opacity=1 ]   (332.48,14.87) -- (312.02,42.16) ;
\draw   (352.79,96.43) -- (363.02,123.7) -- (342.57,123.7) -- cycle ;
\draw [color={rgb, 255:red, 208; green, 2; blue, 27 }  ,draw opacity=1, line width = 0.6mm ]   (332.48,69.93) -- (352.79,96.93) ;
\draw   (312.02,124.01) -- (322.25,151.28) -- (301.8,151.28) -- cycle ;
\draw [color={rgb, 255:red, 0; green, 0; blue, 0 }  ,draw opacity=1 ]   (332.48,96.71) -- (312.02,124.01) ;
\draw   (288.16,42.14) -- (298.39,69.41) -- (277.94,69.41) -- cycle ;
\draw [color={rgb, 255:red, 208; green, 2; blue, 27 }  ,draw opacity=1, line width = 0.6mm ]   (332.61,15.17) -- (288.16,42.14) ;
\draw   (376.79,42.16) -- (387.02,69.43) -- (366.56,69.43) -- cycle ;
\draw [color={rgb, 255:red, 208; green, 2; blue, 27 }  ,draw opacity=1, line width = 0.6mm ]   (332.48,14.87) -- (376.79,42.16) ;
\draw   (376.79,96.73) -- (387.02,123.99) -- (366.56,123.99) -- cycle ;
\draw [color={rgb, 255:red, 208; green, 2; blue, 27 }  ,draw opacity=1, line width = 0.6mm ]   (332.48,69.93) -- (376.79,97.23) ;
\draw   (288.03,123.69) -- (298.25,150.96) -- (277.8,150.96) -- cycle ;
\draw [color={rgb, 255:red, 0; green, 0; blue, 0 }  ,draw opacity=1 ]   (332.48,96.71) -- (288.03,123.69) ;
\draw   (332.48,178.56) -- (342.7,205.83) -- (322.25,205.83) -- cycle ;
\draw [color={rgb, 255:red, 208; green, 2; blue, 27 }  ,draw opacity=1, line width = 0.6mm ]  (332.48,15.37) -- (332.48,69.93) ;
\draw [shift={(332.48,69.93)}, rotate = 90] [color={rgb, 255:red, 208; green, 2; blue, 27 }  ,draw opacity=1][fill={rgb, 255:red, 208; green, 2; blue, 27 }  ,fill opacity=1 ][line width=0.75]      (0, 0) circle [x radius= 2.01, y radius= 2.01]   ;
\draw [shift={(332.48,42.65)}, rotate = 90] [color={rgb, 255:red, 208; green, 2; blue, 27 }  ,draw opacity=1 ][fill={rgb, 255:red, 208; green, 2; blue, 27 }  ,fill opacity=1 ][line width=0.75]      (0, 0) circle [x radius= 2.01, y radius= 2.01]   ;
\draw [shift={(332.48,15.37)}, rotate = 90] [color={rgb, 255:red, 208; green, 2; blue, 27 }  ,draw opacity=1, line width = 0.6mm ][fill={rgb, 255:red, 208; green, 2; blue, 27 }  ,fill opacity=1 ][line width=0.75]      (0, 0) circle [x radius= 2.01, y radius= 2.01]   ;
\draw  [color={rgb, 255:red, 208; green, 2; blue, 27 }  ,draw opacity=1 ][fill={rgb, 255:red, 208; green, 2; blue, 27 }  ,fill opacity=1 ] (310.07,42.16) .. controls (310.07,41.08) and (310.94,40.2) .. (312.02,40.2) .. controls (313.1,40.2) and (313.98,41.08) .. (313.98,42.16) .. controls (313.98,43.24) and (313.1,44.12) .. (312.02,44.12) .. controls (310.94,44.12) and (310.07,43.24) .. (310.07,42.16) -- cycle ;
\draw  [color={rgb, 255:red, 208; green, 2; blue, 27 }  ,draw opacity=1 ][fill={rgb, 255:red, 208; green, 2; blue, 27 }  ,fill opacity=1 ] (350.97,42.16) .. controls (350.97,41.08) and (351.85,40.2) .. (352.93,40.2) .. controls (354.01,40.2) and (354.89,41.08) .. (354.89,42.16) .. controls (354.89,43.24) and (354.01,44.12) .. (352.93,44.12) .. controls (351.85,44.12) and (350.97,43.24) .. (350.97,42.16) -- cycle ;
\draw  [color={rgb, 255:red, 208; green, 2; blue, 27 }  ,draw opacity=1 ][fill={rgb, 255:red, 208; green, 2; blue, 27 }  ,fill opacity=1 ] (374.83,42.16) .. controls (374.83,41.08) and (375.71,40.2) .. (376.79,40.2) .. controls (377.87,40.2) and (378.75,41.08) .. (378.75,42.16) .. controls (378.75,43.24) and (377.87,44.12) .. (376.79,44.12) .. controls (375.71,44.12) and (374.83,43.24) .. (374.83,42.16) -- cycle ;
\draw  [color={rgb, 255:red, 208; green, 2; blue, 27 }  ,draw opacity=1 ][fill={rgb, 255:red, 208; green, 2; blue, 27 }  ,fill opacity=1 ] (350.83,96.43) .. controls (350.83,95.35) and (351.71,94.47) .. (352.79,94.47) .. controls (353.87,94.47) and (354.75,95.35) .. (354.75,96.43) .. controls (354.75,97.51) and (353.87,98.39) .. (352.79,98.39) .. controls (351.71,98.39) and (350.83,97.51) .. (350.83,96.43) -- cycle ;
\draw  [color={rgb, 255:red, 208; green, 2; blue, 27 }  ,draw opacity=1 ][fill={rgb, 255:red, 208; green, 2; blue, 27 }  ,fill opacity=1 ] (374.83,96.73) .. controls (374.83,95.64) and (375.71,94.77) .. (376.79,94.77) .. controls (377.87,94.77) and (378.75,95.64) .. (378.75,96.73) .. controls (378.75,97.81) and (377.87,98.68) .. (376.79,98.68) .. controls (375.71,98.68) and (374.83,97.81) .. (374.83,96.73) -- cycle ;
\draw [color={rgb, 255:red, 0; green, 0; blue, 0 }  ,draw opacity=1 ]   (332.48,123.99) -- (332.48,178.56) ;
\draw [shift={(332.48,178.56)}, rotate = 90] [color={rgb, 255:red, 0; green, 0; blue, 0 }  ,draw opacity=1 ][fill={rgb, 255:red, 0; green, 0; blue, 0 }  ,fill opacity=1 ][line width=0.75]      (0, 0) circle [x radius= 2.01, y radius= 2.01]   ;
\draw [shift={(332.48,151.28)}, rotate = 90] [color={rgb, 255:red, 0; green, 0; blue, 0 }  ,draw opacity=1 ][fill={rgb, 255:red, 0; green, 0; blue, 0 }  ,fill opacity=1 ][line width=0.75]      (0, 0) circle [x radius= 2.01, y radius= 2.01]   ;
\draw [shift={(332.48,123.99)}, rotate = 90] [color={rgb, 255:red, 0; green, 0; blue, 0 }  ,draw opacity=1 ][fill={rgb, 255:red, 0; green, 0; blue, 0 }  ,fill opacity=1 ][line width=0.75]      (0, 0) circle [x radius= 2.01, y radius= 2.01]   ;
\draw  [color={rgb, 255:red, 0; green, 0; blue, 0 }  ,draw opacity=1 ][fill={rgb, 255:red, 0; green, 0; blue, 0 }  ,fill opacity=1 ] (286.07,123.69) .. controls (286.07,122.61) and (286.94,121.73) .. (288.03,121.73) .. controls (289.11,121.73) and (289.98,122.61) .. (289.98,123.69) .. controls (289.98,124.77) and (289.11,125.65) .. (288.03,125.65) .. controls (286.94,125.65) and (286.07,124.77) .. (286.07,123.69) -- cycle ;
\draw  [color={rgb, 255:red, 0; green, 0; blue, 0 }  ,draw opacity=1 ][fill={rgb, 255:red, 0; green, 0; blue, 0 }  ,fill opacity=1 ] (310.07,124.01) .. controls (310.07,122.92) and (310.94,122.05) .. (312.02,122.05) .. controls (313.1,122.05) and (313.98,122.92) .. (313.98,124.01) .. controls (313.98,125.09) and (313.1,125.96) .. (312.02,125.96) .. controls (310.94,125.96) and (310.07,125.09) .. (310.07,124.01) -- cycle ;
\draw  [color={rgb, 255:red, 208; green, 2; blue, 27 }  ,draw opacity=1 ][fill={rgb, 255:red, 208; green, 2; blue, 27 }  ,fill opacity=1 ] (286.2,42.14) .. controls (286.2,41.06) and (287.08,40.18) .. (288.16,40.18) .. controls (289.24,40.18) and (290.12,41.06) .. (290.12,42.14) .. controls (290.12,43.22) and (289.24,44.1) .. (288.16,44.1) .. controls (287.08,44.1) and (286.2,43.22) .. (286.2,42.14) -- cycle ;
\draw  [color={rgb, 255:red, 208; green, 2; blue, 27 }  ,draw opacity=1 ] (394,100) .. controls (398.67,100) and (401,97.67) .. (401,93) -- (401,65.88) .. controls (401,59.21) and (403.33,55.88) .. (408,55.88) .. controls (403.33,55.88) and (401,52.55) .. (401,45.88)(401,48.88) -- (401,18.75) .. controls (401,14.08) and (398.67,11.75) .. (394,11.75) ;
\draw   (180.68,42.06) -- (190.9,69.33) -- (170.45,69.33) -- cycle ;
\draw [color={rgb, 255:red, 74; green, 144; blue, 226 }  ,draw opacity=1, line width = 0.6mm]    (160.76,15.07) -- (181.08,42.06) ;
\draw   (139.77,42.06) -- (150,69.33) -- (129.55,69.33) -- cycle ;
\draw [color={rgb, 255:red, 74; green, 144; blue, 226 }  ,draw opacity=1, line width = 0.6mm]    (160.63,14.77) -- (140.17,42.06) ;
\draw   (180.54,96.33) -- (190.77,123.6) -- (170.32,123.6) -- cycle ;
\draw [color={rgb, 255:red, 74; green, 144; blue, 226 }  ,draw opacity=1, line width = 0.6mm ]   (160.23,69.83) -- (180.54,96.83) ;
\draw   (139.77,123.91) -- (150,151.17) -- (129.55,151.17) -- cycle ;
\draw [color={rgb, 255:red, 74; green, 144; blue, 226 }  ,draw opacity=1, line width = 0.6mm ]   (160.23,97.11) -- (139.77,124.41) ;
\draw   (115.91,42.04) -- (126.14,69.31) -- (105.69,69.31) -- cycle ;
\draw [color={rgb, 255:red, 74; green, 144; blue, 226 }  ,draw opacity=1, line width = 0.6mm ]    (160.36,15.07) -- (115.91,42.04) ;
\draw   (204.54,42.06) -- (214.77,69.33) -- (194.31,69.33) -- cycle ;
\draw [color={rgb, 255:red, 74; green, 144; blue, 226 }  ,draw opacity=1 , line width = 0.6mm]    (160.63,14.77) -- (204.94,42.06) ;
\draw   (204.54,96.62) -- (214.77,123.89) -- (194.31,123.89) -- cycle ;
\draw [color={rgb, 255:red, 74; green, 144; blue, 226 }  ,draw opacity=1, line width = 0.6mm ]   (160.23,69.83) -- (204.54,97.12) ;
\draw   (115.78,123.59) -- (126,150.86) -- (105.55,150.86) -- cycle ;
\draw [color={rgb, 255:red, 74; green, 144; blue, 226 }  ,draw opacity=1, line width = 0.6mm ]   (160.23,97.11) -- (115.78,124.09) ;
\draw   (160.23,178.46) -- (170.45,205.73) -- (150,205.73) -- cycle ;
\draw   (183.68,178.46) -- (193.9,205.73) -- (173.45,205.73) -- cycle ;
\draw [color={rgb, 255:red, 74; green, 144; blue, 226 }  ,draw opacity=1, line width = 0.6mm ]   (160.23,151.67) -- (183.68,178.96) ;
\draw [color={rgb, 255:red, 74; green, 144; blue, 226 }  ,draw opacity=1, line width = 0.6mm ][fill={rgb, 255:red, 74; green, 144; blue, 226 }  ,fill opacity=1 ]  (160.63,14.77) -- (160.63,69.33) ;
\draw [shift={(160.63,69.33)}, rotate = 90] [color={rgb, 255:red, 74; green, 144; blue, 226 }  ,draw opacity=1 ][fill={rgb, 255:red, 74; green, 144; blue, 226 }  ,fill opacity=1 ][line width = 0.75pt]    (0, 0) circle [x radius= 2.01, y radius= 2.01]   ;
\draw [shift={(160.63,42.05)}, rotate = 90] [color={rgb, 255:red, 74; green, 144; blue, 226 }  , draw opacity=1 ][fill={rgb, 255:red, 74; green, 144; blue, 226 }  ,fill opacity=1 ][line width = 0.75pt]     (0, 0) circle [x radius= 2.01, y radius= 2.01]   ;
\draw [shift={(160.63,14.77)}, rotate = 90] [color={rgb, 255:red, 74; green, 144; blue, 226 }  ,draw opacity=1][fill={rgb, 255:red, 74; green, 144; blue, 226 }  ,fill opacity=1 ][line width = 0.75pt]      (0, 0) circle [x radius= 2.01, y radius= 2.01]   ;
\draw  [color={rgb, 255:red, 74; green, 144; blue, 226 }  ,draw opacity=1 ][fill={rgb, 255:red, 74; green, 144; blue, 226 }  ,fill opacity=1 ] (137.82,42.06) .. controls (137.82,40.98) and (138.69,40.1) .. (139.77,40.1) .. controls (140.85,40.1) and (141.73,40.98) .. (141.73,42.06) .. controls (141.73,43.14) and (140.85,44.02) .. (139.77,44.02) .. controls (138.69,44.02) and (137.82,43.14) .. (137.82,42.06) -- cycle ;
\draw  [color={rgb, 255:red, 74; green, 144; blue, 226 }  ,draw opacity=1 ][fill={rgb, 255:red, 74; green, 144; blue, 226 }  ,fill opacity=1 ] (178.72,42.06) .. controls (178.72,40.98) and (179.6,40.1) .. (180.68,40.1) .. controls (181.76,40.1) and (182.64,40.98) .. (182.64,42.06) .. controls (182.64,43.14) and (181.76,44.02) .. (180.68,44.02) .. controls (179.6,44.02) and (178.72,43.14) .. (178.72,42.06) -- cycle ;
\draw  [color={rgb, 255:red, 74; green, 144; blue, 226 }  ,draw opacity=1 ][fill={rgb, 255:red, 74; green, 144; blue, 226 }  ,fill opacity=1 ] (202.58,42.06) .. controls (202.58,40.98) and (203.46,40.1) .. (204.54,40.1) .. controls (205.62,40.1) and (206.5,40.98) .. (206.5,42.06) .. controls (206.5,43.14) and (205.62,44.02) .. (204.54,44.02) .. controls (203.46,44.02) and (202.58,43.14) .. (202.58,42.06) -- cycle ;
\draw  [color={rgb, 255:red, 74; green, 144; blue, 226 }  ,draw opacity=1 ][fill={rgb, 255:red, 74; green, 144; blue, 226 }  ,fill opacity=1 ] (178.58,96.33) .. controls (178.58,95.25) and (179.46,94.37) .. (180.54,94.37) .. controls (181.62,94.37) and (182.5,95.25) .. (182.5,96.33) .. controls (182.5,97.41) and (181.62,98.29) .. (180.54,98.29) .. controls (179.46,98.29) and (178.58,97.41) .. (178.58,96.33) -- cycle ;
\draw  [color={rgb, 255:red, 74; green, 144; blue, 226 }  ,draw opacity=1 ][fill={rgb, 255:red, 74; green, 144; blue, 226 }  ,fill opacity=1 ] (202.58,96.62) .. controls (202.58,95.54) and (203.46,94.67) .. (204.54,94.67) .. controls (205.62,94.67) and (206.5,95.54) .. (206.5,96.62) .. controls (206.5,97.71) and (205.62,98.58) .. (204.54,98.58) .. controls (203.46,98.58) and (202.58,97.71) .. (202.58,96.62) -- cycle ;
\draw [color={rgb, 255:red, 74; green, 144; blue, 226 }  ,draw opacity=1, line width = 0.6mm ]   (160.23,69.83) -- (160.23,124.39) ;
\draw [shift={(160.23,124.39)}, rotate = 90] [color={rgb, 255:red, 74; green, 144; blue, 226 }  ,draw opacity=1 ][fill={rgb, 255:red, 74; green, 144; blue, 226 }  ,fill opacity=1 ][line width=0.75]      (0, 0) circle [x radius= 2.01, y radius= 2.01]   ;
\draw [shift={(160.23,97.11)}, rotate = 90] [color={rgb, 255:red, 74; green, 144; blue, 226 }  ,draw opacity=1 ][fill={rgb, 255:red, 74; green, 144; blue, 226 }  ,fill opacity=1 ][line width=0.75]      (0, 0) circle [x radius= 2.01, y radius= 2.01]   ;
\draw [shift={(160.23,69.83)}, rotate = 90] [color={rgb, 255:red, 74; green, 144; blue, 226 }  ,draw opacity=1 ][fill={rgb, 255:red, 74; green, 144; blue, 226 }  ,fill opacity=1 ][line width=0.75]      (0, 0) circle [x radius= 2.01, y radius= 2.01]   ;
\draw [color={rgb, 255:red, 74; green, 144; blue, 226 }  ,draw opacity=1, line width = 0.6mm ]   (160.23,124.39) -- (160.23,178.96) ;
\draw [shift={(160.23,178.96)}, rotate = 90] [color={rgb, 255:red, 74; green, 144; blue, 226 }  ,draw opacity=1 ][fill={rgb, 255:red, 74; green, 144; blue, 226 }  ,fill opacity=1 ][line width=0.75]      (0, 0) circle [x radius= 2.01, y radius= 2.01]   ;
\draw [shift={(160.23,151.67)}, rotate = 90] [color={rgb, 255:red, 74; green, 144; blue, 226 }  ,draw opacity=1 ][fill={rgb, 255:red, 74; green, 144; blue, 226 }  ,fill opacity=1 ][line width=0.75]      (0, 0) circle [x radius= 2.01, y radius= 2.01]   ;
\draw [shift={(160.23,124.39)}, rotate = 90] [color={rgb, 255:red, 74; green, 144; blue, 226 }  ,draw opacity=1 ][fill={rgb, 255:red, 74; green, 144; blue, 226 }  ,fill opacity=1 ][line width=0.75]      (0, 0) circle [x radius= 2.01, y radius= 2.01]   ;
\draw  [color={rgb, 255:red, 74; green, 144; blue, 226 }  ,draw opacity=1 ][fill={rgb, 255:red, 74; green, 144; blue, 226 }  ,fill opacity=1 ] (113.82,123.59) .. controls (113.82,122.51) and (114.69,121.63) .. (115.78,121.63) .. controls (116.86,121.63) and (117.73,122.51) .. (117.73,123.59) .. controls (117.73,124.67) and (116.86,125.55) .. (115.78,125.55) .. controls (114.69,125.55) and (113.82,124.67) .. (113.82,123.59) -- cycle ;
\draw  [color={rgb, 255:red, 74; green, 144; blue, 226 }  ,draw opacity=1 ][fill={rgb, 255:red, 74; green, 144; blue, 226 }  ,fill opacity=1 ] (137.82,123.91) .. controls (137.82,122.82) and (138.69,121.95) .. (139.77,121.95) .. controls (140.85,121.95) and (141.73,122.82) .. (141.73,123.91) .. controls (141.73,124.99) and (140.85,125.86) .. (139.77,125.86) .. controls (138.69,125.86) and (137.82,124.99) .. (137.82,123.91) -- cycle ;
\draw  [color={rgb, 255:red, 74; green, 144; blue, 226 }  ,draw opacity=1 ][fill={rgb, 255:red, 74; green, 144; blue, 226 }  ,fill opacity=1 ] (181.72,178.46) .. controls (181.72,177.37) and (182.6,176.5) .. (183.68,176.5) .. controls (184.76,176.5) and (185.64,177.37) .. (185.64,178.46) .. controls (185.64,179.54) and (184.76,180.41) .. (183.68,180.41) .. controls (182.6,180.41) and (181.72,179.54) .. (181.72,178.46) -- cycle ;
\draw  [color={rgb, 255:red, 74; green, 144; blue, 226 }  ,draw opacity=1 ][fill={rgb, 255:red, 74; green, 144; blue, 226 }  ,fill opacity=1 ] (113.95,42.04) .. controls (113.95,40.96) and (114.83,40.08) .. (115.91,40.08) .. controls (116.99,40.08) and (117.87,40.96) .. (117.87,42.04) .. controls (117.87,43.12) and (116.99,44) .. (115.91,44) .. controls (114.83,44) and (113.95,43.12) .. (113.95,42.04) -- cycle ;
\draw  [color={rgb, 255:red, 74; green, 144; blue, 226 }  ,draw opacity=1 ] (99.25,12.5) .. controls (94.58,12.5) and (92.25,14.83) .. (92.25,19.5) -- (92.25,86.38) .. controls (92.25,93.05) and (89.92,96.38) .. (85.25,96.38) .. controls (89.92,96.38) and (92.25,99.71) .. (92.25,106.38)(92.25,103.38) -- (92.25,173.25) .. controls (92.25,177.92) and (94.58,180.25) .. (99.25,180.25) ;
\draw   (355.93,178.56) -- (366.15,205.83) -- (345.7,205.83) -- cycle ;
\draw [color={rgb, 255:red, 0; green, 0; blue, 0 }  ,draw opacity=1 ]   (332.48,151.28) -- (355.93,178.56) ;
\draw  [color={rgb, 255:red, 0; green, 0; blue, 0 }  ,draw opacity=1 ][fill={rgb, 255:red, 0; green, 0; blue, 0 }  ,fill opacity=1 ] (353.97,178.56) .. controls (353.97,177.48) and (354.85,176.6) .. (355.93,176.6) .. controls (357.01,176.6) and (357.89,177.48) .. (357.89,178.56) .. controls (357.89,179.64) and (357.01,180.52) .. (355.93,180.52) .. controls (354.85,180.52) and (353.97,179.64) .. (353.97,178.56) -- cycle ;
\draw  [color={rgb, 255:red, 208; green, 2; blue, 27 }  ,draw opacity=1 ][fill={rgb, 255:red, 208; green, 2; blue, 27 }  ,fill opacity=1 ] (330.38,96.71) .. controls (330.38,95.56) and (331.32,94.62) .. (332.48,94.62) .. controls (333.63,94.62) and (334.57,95.56) .. (334.57,96.71) .. controls (334.57,97.87) and (333.63,98.81) .. (332.48,98.81) .. controls (331.32,98.81) and (330.38,97.87) .. (330.38,96.71) -- cycle ;
\draw [color={rgb, 255:red, 208; green, 2; blue, 27 }  ,draw opacity=1, line width = 0.6mm ]   (332.48,69.93) -- (332.48,97.21) ;

\draw (354.32,9.85) node   [align=left] {\begin{minipage}[lt]{19.97pt}\setlength\topsep{0pt}
$\displaystyle r( t)$
\end{minipage}};
\draw (434.5,57.6) node  [color={rgb, 255:red, 208; green, 2; blue, 27 }  ,opacity=1 ] [align=left] {\begin{minipage}[lt]{33.67pt}\setlength\topsep{0pt}
$\displaystyle S_{k}( t,v)$
\end{minipage}};
\draw (322.57,179) node   [align=left] {\begin{minipage}[lt]{10.45pt}\setlength\topsep{0pt}
$\displaystyle v$
\end{minipage}};
\draw (182.07,9.75) node   [align=left] {\begin{minipage}[lt]{19.97pt}\setlength\topsep{0pt}
$\displaystyle r( t)$
\end{minipage}};
\draw (62,97) node  [color={rgb, 255:red, 74; green, 144; blue, 226 }  ,opacity=1 ] [align=left] {\begin{minipage}[lt]{33.67pt}\setlength\topsep{0pt}
$\displaystyle S( t,v)$
\end{minipage}};
\draw (150.32,178.9) node   [align=left] {\begin{minipage}[lt]{10.45pt}\setlength\topsep{0pt}
$\displaystyle v$
\end{minipage}};
\draw (318.07,94) node   [align=left] {\begin{minipage}[lt]{10.45pt}\setlength\topsep{0pt}
$\displaystyle v^{k}$
\end{minipage}};

\end{tikzpicture}
    \caption{A visualization of a spine $S(t, v)$ and of a $k$-spine $S_k(t,v)$, for $k=3$.}\label{fig:spine}
\end{figure}
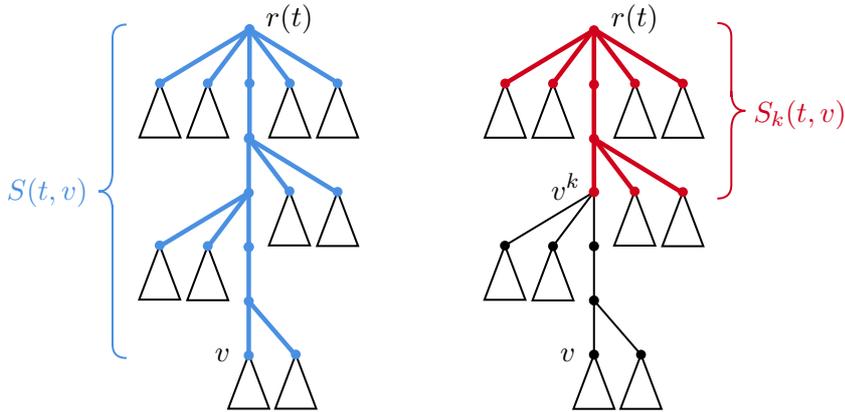

Some of the definitions of the coming paragraph are illustrated in Figure~\ref{fig:spine}. 
For nodes $x,y$ of a tree $t$, we write $x \prec y$ if $x$ is an ancestor of $y$ in $t$, and for a node $z\in v(t) \setminus \{r(t)\}$, we denote its parent by $p(z)$. 
Given a marked tree $(t,v)$, the {\em spine} $\spine(t,v)$ of $(t,v)$ is the subtree of $t$ with vertices $\{w: p(w) \prec v\} \cup \{r(t)\}$. 
For $0 \le k \le |v|$, write $v^k$ for the unique ancestor of $v$ with $|v^k|=k$; so $v^0=r(t)$ and $v^{|v|}=v$. If $k \le |v|$ then 
the {\em $k$-spine} $\spine_k(t,v)$ of $(t,v)$ is the subtree of $t$ with vertices $\{w: p(w) \prec v^k\}\cup \{r(t)\}$, and the {\em marked $k$-spine} of $(t,v)$ is the marked tree $(\spine_k(t,v) ,v^k)$.  
The \textit{spinal degree sequence} of $\spine_k(t,v)$ is $(d_t(v^0),\ldots,d_t(v^{k-1}))$. 
Given a sequence $\dseq=(d_0,\ldots,d_{k-1})$ of non-negative integers and degree statistics $\dstat=(\dstat(c),c \ge 0)$ satisfying $\sum_{c \ge 0} c\dstat(c) = \sum_{c \ge 0} \dstat(c) - 1$, write 
\[
\tset_\dstat^\bullet(\dseq) = \{(t,v) \in \tset_\dstat^\bullet: |v| \ge k, (\deg_t(v^0),\ldots,\deg_t(v^{k-1}))=\dseq\}\, .
\]

Using these definitions, we establish the following combinatorial result, whose probabilistic corollary is then used to prove Proposition~\ref{prop:sampler}. (This result is closely related to the backbone decomposition of trees given in \cite[Proposition 4 (a)]{MR3188597} in order to prove convergence of large random trees with fixed degrees to the Brownian continuum random tree after rescaling, under suitable assumptions on the degree sequences.)
\begin{prop}\label{prop:tnd_count} 
Fix degree statistics $\dstat=(\dstat(c),c \ge 0)$ with $\sum_{c \ge 0} \dstat(c)=n$ and with $\sum_{c \ge 0} c\dstat(c)=n-1$.
For any sequence $\dseq=(d_0,\ldots,d_{k-1})$ of non-negative integers, let $w(\dseq) = (w(\dseq,c), c\geq 0)$ where $w(\dseq,c) = |\{0 \le i \le k-1: d_i = c\}|$ is as in \eqref{eq:size_biasing}. If $w(\dseq,c) \le \dstat(c)$ for all $c \ge 0$, then
\[
|\tset_\dstat^\bullet(\dseq)| = \pran{\prod_{i=0}^{k-1} d_i} \cdot\left|\tset_{\dstat-w(\dseq)}^{(1)}\right| = \prod_{i=0}^{k-1} d_i \cdot {n-k \choose \dstat(c)-w(\dseq,c),c \ge 0}\, .
\]
\end{prop}
\begin{proof}
To describe an element $(t,v)$ of $\tset_\dstat^\bullet(\dseq)$, it is necessary and sufficient to specify the marked $k$-spine $(\spine_k(t,v) ,v^k)$, the subtrees of $t$ rooted at the leaves of $\spine_k(t,v)$, and the identity of the mark $v$, which must lie within the subtree of $\spine_k(t,v)$ rooted at $v^k$. 

The number of marked $k$-spines with spinal degree sequence $\dseq$ is $\prod_{i=0}^{k-1} d_i$, since to specify such a tree it is necessary and sufficient to indicate which of the $d_i$ children of $v^i$ is $v^{i+1}$ for each $0 \le i \le k-1$. The subtrees rooted at the leaves of $\spine_k(t,v)$ form a rooted forest with degree statistics $\dstat-w(\dseq)$, with a marked vertex in a specific tree; by (\ref{eq:marked_forest_count}) the number of such marked forests is 
\[
{n-k \choose \dstat(c)-w(\dseq,c),c \ge 0}.
\]
The result follows.
\end{proof}
For the next corollary we introduce the falling factorial notation $(m)_b = m(m-1)\cdot \ldots\cdot (m-b+1)$. 
\begin{cor}\label{cor:dseq_formula}
Fix degree statistics $\dstat=(\dstat(c),c \ge 0)$ with $\sum_{c \ge 0} \dstat(c)=n$ and with $\sum_{c \ge 0} c\dstat(c)=n-1$.
For any sequence $\dseq=(d_0,\ldots,d_{k-1})$ of non-negative integers with $w(\dseq, c)\le \dstat(c)$ for all $c\ge 0$, if $(T,V) \in_u \tset_\dstat^\bullet$ then we have
\[
\p{\big(\deg_t(V^0),\ldots,\deg_t(V^{k-1})\big)=\dseq,|V| \ge k}
= 
\frac{1}{(n)_k} \cdot \prod_{i=0}^{k-1} d_i\cdot \prod_{c \ge 0} (\dstat(c))_{w(\dseq,c)}\, .
\]
\end{cor}
\begin{proof}
Since $(T,V) \in_u \tset_\dstat^\bullet$, this probability is simply $|\tset_\dstat^\bullet(\dseq)|/|\tset_\dstat^\bullet|$, and the result follows from the formula for $|\tset_\dstat^\bullet(\dseq)|$ given in  Proposition~\ref{prop:tnd_count}. 
\end{proof}

\begin{proof}[Proof of Proposition~\ref{prop:sampler}]
Write $n=\sum_{c \ge 0}\dstat(c)$. 
It suffices to show that for any sequence of non-negative integers $\dseq=(d_1,\ldots,d_{k})$ with $w(\dseq, c)\le \dstat(c)$ for all $c\ge 0$, 
\begin{equation}\label{eq:sampler_toprove}
\p{(D_1,\ldots,D_k)=\dseq,M \ge k+1}
=\frac{1}{(n)_k}\cdot\prod_{i=1}^{k} d_i \cdot \prod_{c \ge 0} (\dstat(c))_{w(\dseq,c)}\, ,
\end{equation}
since then by summing over $\dseq$ in this equation and in Corollary~\ref{cor:dseq_formula} (note the shift by $1$ of the indices) it follows that $\p{|V| \ge k} = \p{M \ge k+1}$, which establishes the distributional identity. 

By the defining equation (\ref{eq:size_biasing}) for the size-biased sequence $(D_1,\ldots,D_n)$ and since $|\dstat|_1=n-1$, we have 
\begin{align}
\p{(D_1,\ldots,D_k)=\dseq}
& = 
\prod_{i=1}^k \probC{D_i=d_i}{(D_1,\ldots,D_{i-1})=(d_1,\ldots,d_{i-1})}\nonumber\\
& = \prod_{i=1}^k
\frac{d_{i}(\dstat(d_i)-w((d_1,\ldots,d_{i-1}),d_i))}{n-1-d_1-\ldots-d_{i-1}}.
\label{eq:d1todk_ident}
\end{align}
On the event that $(D_1,\ldots,D_{k})=(d_1,\ldots,d_k)$, we have $M \ge k+1$ if and only if 
\[
U_i > \frac{1+\sum_{j=1}^{i-1} (d_{j}-1)}{n+1-i}
\]
for each $1 \le i \le k$. 
Since 
\[
n+1-i-\Big(1+\sum_{j=1}^{i-1} (d_j-1)\Big)
= n-1-d_1-\ldots-d_{i-1}\, ,
\]
it follows that 
\begin{align*}
\probC{M \ge k+1}{(D_1,\ldots,D_k)=\dseq}
& = \prod_{i=1}^{k} \p{U_i > \frac{1+\sum_{j=1}^{i-1} (d_{j}-1)}{n+1-i}}
\\
& = \prod_{i=1}^{k} \frac{n-1-d_1-\ldots-d_{i-1}}{n+1-i}\, ,
\end{align*}
so by (\ref{eq:d1todk_ident}),
\begin{align*}
 \p{(D_1,\ldots,D_k)=\dseq,M \ge k+1}
 & = 
\prod_{i=1}^k 
\frac{d_{i}(\dstat(d_i)-w((d_1,\ldots,d_{i-1}), d_i)}{n+1-i} \\
& = \frac{1}{(n)_k} \prod_{i=1}^k d_i (\dstat(d_i)-w((d_1,\ldots,d_{i-1}), d_i). 
\end{align*}
Equation (\ref{eq:sampler_toprove}) follows since
\[
\prod_{i=1}^k (\dstat(d_i)-w((d_1,\ldots,d_{i-1}),d_i)) = \prod_{c \ge 0} (\dstat(c))_{w(\dseq,c)}\, . \qedhere
\]
\end{proof}

\section{\bf Proof of Theorem~\ref{thm:main_mod}} 

Fix degree statistics $\dstat=(\dstat(c),c \ge 0)$ with $\sum_{c \ge 0} \dstat(c) = n$ and $|\dstat|_1=\sum_{c \ge 0} c\dstat(c)=n-1$. 
To prove the theorem, we construct a size-biasing of $\dstat$ using a Poissonization trick similar to one introduced in \cite{MR1741774} in the context of inhomogeneous continuum random trees. Several of the definitions of the next two paragraphs are illustrated in Figure~\ref{fig:poisson}.

Let $(d_1,\ldots,d_n)$ be such that 
$d_i = c$ if and only if $\sum_{b=0}^{c-1} \dstat(b) < i \le \sum_{b=0}^c \dstat(b)$, so that $(d_1,\ldots,d_n)$ contains $\dstat(c)$ entries with value $c$ for each $c \ge 0$. Let $l_1=0$ and for $1 \le i \le n$, let 
$l_{i+1} = l_i + d_i/(n-1)$, and define $I_i = [l_i,l_{i+1})$. The intervals $I_1,\ldots,I_n$ are disjoint and partition $[0,1)$. 

Next let $\bN$ be a homogeneous Poisson process on $[0,\infty) \times[0,1)$, with atoms $((S_\ell,U_\ell),\ell \ge 1)$ listed in increasing order of arrival time (so $0 < S_1 < S_2 < \ldots$). Then $(S_\ell,\ell \ge 1)$ is a rate-one Poisson process on $[0,\infty)$, and the random variables $(U_\ell,\ell \ge 1)$ are independent Uniform$[0,1]$, independent of $(S_\ell,\ell \ge 1)$. 

For each $\ell \ge 1$, let $J(\ell)$ be the index of the interval containing the point $U_\ell$, so $U_\ell \in I_{J(\ell)}$. Let $M(1) = 1$, and for $\ell \ge 1$ let 
\[
M(\ell+1) = \inf\{k> M(\ell): U_k \not\in I_{J(M(1))}\cup\ldots\cup I_{J(M(\ell))}\}\, ,
\]
so $i \in \{M(\ell),\ell \ge 1\}$ precisely if $U_i \not\in I_1\cup\ldots \cup I_{i-1}$. 
Then for any vector $(j(1),\ldots,j(\ell))$ of distinct elements of $[n]$ and any increasing sequence of  integers $(m(1),\ldots,m(\ell))$ with $m(1)=1$, we have 
\begin{align*}
& \probC{J(m(\ell))=j(\ell),M(\ell)=m(\ell)}{J(m(k))=j(k)\mbox{ and }M(k)=m(k)\mbox{ for all }1 \le k \le \ell-1}\\
& = 
\p{U_{m(\ell)} \in I_{j(\ell)}\mbox{ and } U_j \in \bigcup_{k=1}^{\ell-1} I_{j(k)}\mbox{ for all }j \in \{m(\ell-1)+1,\ldots,m(\ell)-1\}}\\
& = 
\frac{d_{j(\ell)}}{n-1} \cdot \pran{\frac{\sum_{k=1}^{\ell-1} d_{j(k)}}{n-1}}^{m(\ell)-m(\ell-1)-1}. 
\end{align*}
Summing over all possible values for $(M(1),\ldots,M(\ell))$, it follows that if $\sum_{k=1}^{\ell-1}d_{j(k)}< n-1$ then 
\[
\probC{J(M(\ell))=j(\ell)}{J(M(k))=j(k)\mbox{ for all } 1 \le k \le \ell-1} = \frac{d_{j(\ell)}}{n-1-d_{j(1)}-\ldots-d_{j(\ell-1)}}\, .
\]
This implies that, writing $n'=n-\dstat(0)$,
the sequence $(D(1),\ldots,D(n))$ defined by 
\[
D(\ell) = \begin{cases}
		d_{J(M(\ell))}	& \mbox{ if }1 \le \ell \le n'\\
		0		& \mbox{ otherwise}
		\end{cases}
\]
is a size-biasing of $\dstat$, since 
\begin{align*}
	& \p{D(l) = d\mid (D(1),...,D(l-1)) = (d_1,...,d_{l-1})} \\
	& = \probC{d_{J(M(l))} = d}{(D(1),...,D(l-1)) = (d_1,...,d_{l-1})}\\
	& = \sum_{i\in [n]} \probC{J(M(l))= i,\, d_i = d }{(D(1),...,D(l-1)) = (d_1,...,d_{l-1})}\\
	& = \frac{d(\dstat(d) - w((d_1,...,d_{l-1}),d)}{n - 1 - d_1 - ... - d_{l-1}},
\end{align*}
where we recall that $w((d_1,...,d_{l-1}),d) = |\{i\in [l-1]: d_i = d\}|$. 

When parsing the definitions of the coming paragraph, Figure~\ref{fig:poisson} will again be useful. For each $1 \le i \le n$, let $r_i = l_i + \max(0,(d_i-1)/(n-1)) \le l_{i+1}$. For $\ell \ge 1$ let $C_\ell = \bigcup_{k=1}^\ell [r_{J(k)},l_{J(k)+1})$, so for each interval $I_i$ which contains at least one point from $U_1,\ldots,U_\ell$, the region $C_\ell$ contains a sub-interval of $I_i$ of length $1/(n-1)$.
Let
\[
\tau = \min\pran{\ell\ge 1: U_\ell \in \bigcup_{k=1}^{\ell-1} I_{J(k)}\setminus C_{\ell-1}} = 
\min\pran{\ell\ge 1: U_\ell \in \bigcup_{k=1}^{\ell-1} [l_{J(k)},r_{J(k)})}\, . 
\]

By the definition of the indices $(M(\ell),\ell \ge 1)$, we have $\tau \not \in \{M(\ell),\ell \ge 1)\}$, since the points $(U_{M(\ell)},\ell \ge 1)$ are precisely those  which, on their arrival, land in a previously empty interval, whereas $U_\tau$ falls in a subinterval of an interval which already contains one of $U_1,\ldots,U_{\tau-1}$. 

\begin{figure}[hbt]
    \centering

 
\tikzset{
pattern size/.store in=\mcSize, 
pattern size = 5pt,
pattern thickness/.store in=\mcThickness, 
pattern thickness = 0.3pt,
pattern radius/.store in=\mcRadius, 
pattern radius = 1pt}
\makeatletter
\pgfutil@ifundefined{pgf@pattern@name@_jy889d05m}{
\pgfdeclarepatternformonly[\mcThickness,\mcSize]{_jy889d05m}
{\pgfqpoint{0pt}{0pt}}
{\pgfpoint{\mcSize+\mcThickness}{\mcSize+\mcThickness}}
{\pgfpoint{\mcSize}{\mcSize}}
{
\pgfsetcolor{\tikz@pattern@color}
\pgfsetlinewidth{\mcThickness}
\pgfpathmoveto{\pgfqpoint{0pt}{0pt}}
\pgfpathlineto{\pgfpoint{\mcSize+\mcThickness}{\mcSize+\mcThickness}}
\pgfusepath{stroke}
}}
\makeatother

 
\tikzset{
pattern size/.store in=\mcSize, 
pattern size = 5pt,
pattern thickness/.store in=\mcThickness, 
pattern thickness = 0.3pt,
pattern radius/.store in=\mcRadius, 
pattern radius = 1pt}
\makeatletter
\pgfutil@ifundefined{pgf@pattern@name@_perv8pg8o}{
\pgfdeclarepatternformonly[\mcThickness,\mcSize]{_perv8pg8o}
{\pgfqpoint{0pt}{0pt}}
{\pgfpoint{\mcSize+\mcThickness}{\mcSize+\mcThickness}}
{\pgfpoint{\mcSize}{\mcSize}}
{
\pgfsetcolor{\tikz@pattern@color}
\pgfsetlinewidth{\mcThickness}
\pgfpathmoveto{\pgfqpoint{0pt}{0pt}}
\pgfpathlineto{\pgfpoint{\mcSize+\mcThickness}{\mcSize+\mcThickness}}
\pgfusepath{stroke}
}}
\makeatother

 
\tikzset{
pattern size/.store in=\mcSize, 
pattern size = 5pt,
pattern thickness/.store in=\mcThickness, 
pattern thickness = 0.3pt,
pattern radius/.store in=\mcRadius, 
pattern radius = 1pt}
\makeatletter
\pgfutil@ifundefined{pgf@pattern@name@_tm4jdcxlk}{
\pgfdeclarepatternformonly[\mcThickness,\mcSize]{_tm4jdcxlk}
{\pgfqpoint{0pt}{0pt}}
{\pgfpoint{\mcSize+\mcThickness}{\mcSize+\mcThickness}}
{\pgfpoint{\mcSize}{\mcSize}}
{
\pgfsetcolor{\tikz@pattern@color}
\pgfsetlinewidth{\mcThickness}
\pgfpathmoveto{\pgfqpoint{0pt}{0pt}}
\pgfpathlineto{\pgfpoint{\mcSize+\mcThickness}{\mcSize+\mcThickness}}
\pgfusepath{stroke}
}}
\makeatother

 
\tikzset{
pattern size/.store in=\mcSize, 
pattern size = 5pt,
pattern thickness/.store in=\mcThickness, 
pattern thickness = 0.3pt,
pattern radius/.store in=\mcRadius, 
pattern radius = 1pt}
\makeatletter
\pgfutil@ifundefined{pgf@pattern@name@_nlv3f6ai7}{
\pgfdeclarepatternformonly[\mcThickness,\mcSize]{_nlv3f6ai7}
{\pgfqpoint{0pt}{0pt}}
{\pgfpoint{\mcSize+\mcThickness}{\mcSize+\mcThickness}}
{\pgfpoint{\mcSize}{\mcSize}}
{
\pgfsetcolor{\tikz@pattern@color}
\pgfsetlinewidth{\mcThickness}
\pgfpathmoveto{\pgfqpoint{0pt}{0pt}}
\pgfpathlineto{\pgfpoint{\mcSize+\mcThickness}{\mcSize+\mcThickness}}
\pgfusepath{stroke}
}}
\makeatother
\tikzset{every picture/.style={line width=0.75pt}} 

\begin{tikzpicture}[x=0.75pt,y=0.75pt,yscale=-1,xscale=1]

\draw  [color={rgb, 255:red, 66; green, 118; blue, 179 }  ,draw opacity=1 ][pattern=_jy889d05m,pattern size=6pt,pattern thickness=0.75pt,pattern radius=0pt, pattern color={rgb, 255:red, 74; green, 144; blue, 226}] (137.84,145.03) -- (507.83,145.03) -- (507.83,153.02) -- (137.84,153.02) -- cycle ;
\draw [color={rgb, 255:red, 0; green, 0; blue, 0 }  ,draw opacity=1 ]   (87.72,52.85) -- (508.82,52.85) ;
\draw [color={rgb, 255:red, 74; green, 144; blue, 226 }  ,draw opacity=1 ]   (92.28,182.75) -- (507.45,182.75) ;
\draw [color={rgb, 255:red, 74; green, 144; blue, 226 }  ,draw opacity=1 ]   (92.28,145.03) -- (507.83,145.03) ;
\draw [color={rgb, 255:red, 74; green, 144; blue, 226 }  ,draw opacity=1 ]   (92.28,118.1) -- (507.83,118.1) ;
\draw [color={rgb, 255:red, 74; green, 144; blue, 226 }  ,draw opacity=1 ]   (92.18,95.27) -- (507.54,95.27) ;
\draw [color={rgb, 255:red, 74; green, 144; blue, 226 }  ,draw opacity=1 ]   (92.28,74.6) -- (507.83,74.6) ;
\draw [color={rgb, 255:red, 139; green, 87; blue, 42 }  ,draw opacity=1 ] [dash pattern={on 4.5pt off 4.5pt}]  (137.2,171.17) -- (137.2,233.39) ;
\draw [color={rgb, 255:red, 139; green, 87; blue, 42 }  ,draw opacity=1 ] [dash pattern={on 4.5pt off 4.5pt}]  (181.8,131.47) -- (181.8,233.39) ;
\draw [color={rgb, 255:red, 139; green, 87; blue, 42 }  ,draw opacity=1 ] [dash pattern={on 4.5pt off 4.5pt}]  (137.2,171.17) -- (94.79,171.17) ;
\draw [color={rgb, 255:red, 139; green, 87; blue, 42 }  ,draw opacity=1 ] [dash pattern={on 4.5pt off 4.5pt}]  (181.8,131.47) -- (104.57,131.47) ;
\draw  [color={rgb, 255:red, 0; green, 0; blue, 0 }  ,draw opacity=1 ][fill={rgb, 255:red, 0; green, 0; blue, 0 }  ,fill opacity=1 ] (134.48,171.17) .. controls (134.48,169.67) and (135.7,168.45) .. (137.2,168.45) .. controls (138.71,168.45) and (139.92,169.67) .. (139.92,171.17) .. controls (139.92,172.67) and (138.71,173.89) .. (137.2,173.89) .. controls (135.7,173.89) and (134.48,172.67) .. (134.48,171.17) -- cycle ;
\draw  [color={rgb, 255:red, 0; green, 0; blue, 0 }  ,draw opacity=1 ][fill={rgb, 255:red, 0; green, 0; blue, 0 }  ,fill opacity=1 ] (179.08,131.47) .. controls (179.08,129.97) and (180.3,128.76) .. (181.8,128.76) .. controls (183.3,128.76) and (184.52,129.97) .. (184.52,131.47) .. controls (184.52,132.98) and (183.3,134.19) .. (181.8,134.19) .. controls (180.3,134.19) and (179.08,132.98) .. (179.08,131.47) -- cycle ;
\draw  [color={rgb, 255:red, 0; green, 0; blue, 0 }  ,draw opacity=1 ][fill={rgb, 255:red, 0; green, 0; blue, 0 }  ,fill opacity=1 ] (266.09,204.89) .. controls (266.09,203.39) and (267.31,202.17) .. (268.81,202.17) .. controls (270.31,202.17) and (271.53,203.39) .. (271.53,204.89) .. controls (271.53,206.39) and (270.31,207.61) .. (268.81,207.61) .. controls (267.31,207.61) and (266.09,206.39) .. (266.09,204.89) -- cycle ;
\draw  [color={rgb, 255:red, 0; green, 0; blue, 0 }  ,draw opacity=1 ][fill={rgb, 255:red, 0; green, 0; blue, 0 }  ,fill opacity=1 ] (309.59,66.81) .. controls (309.59,65.3) and (310.81,64.09) .. (312.31,64.09) .. controls (313.81,64.09) and (315.03,65.3) .. (315.03,66.81) .. controls (315.03,68.31) and (313.81,69.52) .. (312.31,69.52) .. controls (310.81,69.52) and (309.59,68.31) .. (309.59,66.81) -- cycle ;
\draw  [color={rgb, 255:red, 0; green, 0; blue, 0 }  ,draw opacity=1 ][fill={rgb, 255:red, 0; green, 0; blue, 0 }  ,fill opacity=1 ] (397.15,109.22) .. controls (397.15,107.72) and (398.37,106.5) .. (399.87,106.5) .. controls (401.37,106.5) and (402.59,107.72) .. (402.59,109.22) .. controls (402.59,110.72) and (401.37,111.94) .. (399.87,111.94) .. controls (398.37,111.94) and (397.15,110.72) .. (397.15,109.22) -- cycle ;
\draw  [color={rgb, 255:red, 0; green, 0; blue, 0 }  ,draw opacity=1 ][fill={rgb, 255:red, 0; green, 0; blue, 0 }  ,fill opacity=1 ] (435.76,109.18) .. controls (435.76,107.68) and (436.98,106.46) .. (438.48,106.46) .. controls (439.98,106.46) and (441.2,107.68) .. (441.2,109.18) .. controls (441.2,110.68) and (439.98,111.9) .. (438.48,111.9) .. controls (436.98,111.9) and (435.76,110.68) .. (435.76,109.18) -- cycle ;
\draw [color={rgb, 255:red, 139; green, 87; blue, 42 }  ,draw opacity=1 ] [dash pattern={on 4.5pt off 4.5pt}]  (486.33,85.25) -- (486.33,232.31) ;
\draw [color={rgb, 255:red, 139; green, 87; blue, 42 }  ,draw opacity=1 ] [dash pattern={on 4.5pt off 4.5pt}]  (486.33,87.97) -- (96.42,87.97) ;
\draw  [color={rgb, 255:red, 0; green, 0; blue, 0 }  ,draw opacity=1 ][fill={rgb, 255:red, 0; green, 0; blue, 0 }  ,fill opacity=1 ] (483.61,87.97) .. controls (483.61,86.47) and (484.83,85.25) .. (486.33,85.25) .. controls (487.84,85.25) and (489.05,86.47) .. (489.05,87.97) .. controls (489.05,89.47) and (487.84,90.69) .. (486.33,90.69) .. controls (484.83,90.69) and (483.61,89.47) .. (483.61,87.97) -- cycle ;
\draw [color={rgb, 255:red, 139; green, 87; blue, 42 }  ,draw opacity=1 ]   (268.81,221.42) -- (268.81,233.39) ;
\draw [color={rgb, 255:red, 139; green, 87; blue, 42 }  ,draw opacity=1 ]   (181.8,221.43) -- (181.8,233.39) ;
\draw [color={rgb, 255:red, 139; green, 87; blue, 42 }  ,draw opacity=1 ]   (137.2,221.43) -- (137.2,233.39) ;
\draw [color={rgb, 255:red, 139; green, 87; blue, 42 }  ,draw opacity=1 ]   (486.33,221.43) -- (486.33,233.39) ;
\draw [color={rgb, 255:red, 139; green, 87; blue, 42 }  ,draw opacity=1 ]   (104.03,171.17) -- (92.61,171.17) ;
\draw [color={rgb, 255:red, 139; green, 87; blue, 42 }  ,draw opacity=1 ]   (103.49,131.47) -- (92.07,131.47) ;
\draw [color={rgb, 255:red, 139; green, 87; blue, 42 }  ,draw opacity=1 ]   (103.49,87.97) -- (92.07,87.97) ;
\draw  [color={rgb, 255:red, 66; green, 118; blue, 179 }  ,draw opacity=1 ][fill={rgb, 255:red, 74; green, 144; blue, 226 }  ,fill opacity=0.16 ] (181.8,118.1) -- (507.83,118.1) -- (507.83,126.08) -- (181.8,126.08) -- cycle ;
\draw    (98.59,38.71) -- (98.59,238.83) ;
\draw    (87.72,227.19) -- (520.77,227.19) ;
\draw [shift={(522.77,227.19)}, rotate = 180] [color={rgb, 255:red, 0; green, 0; blue, 0 }  ][line width=0.75]    (10.93,-4.9) .. controls (6.95,-2.3) and (3.31,-0.67) .. (0,0) .. controls (3.31,0.67) and (6.95,2.3) .. (10.93,4.9)   ;
\draw  [color={rgb, 255:red, 74; green, 144; blue, 226 }  ,draw opacity=1 ][fill={rgb, 255:red, 74; green, 144; blue, 226 }  ,fill opacity=0.16 ] (399.87,95.27) -- (507.54,95.27) -- (507.54,103.25) -- (399.87,103.25) -- cycle ;
\draw  [color={rgb, 255:red, 66; green, 118; blue, 179 }  ,draw opacity=1 ][fill={rgb, 255:red, 74; green, 144; blue, 226 }  ,fill opacity=0.16 ] (312.49,52.85) -- (508.82,52.85) -- (508.82,60.83) -- (312.49,60.83) -- cycle ;
\draw  [color={rgb, 255:red, 74; green, 144; blue, 226 }  ,draw opacity=1 ][fill={rgb, 255:red, 74; green, 144; blue, 226 }  ,fill opacity=0.16 ] (486.33,74.6) -- (507.83,74.6) -- (507.83,82.59) -- (486.33,82.59) -- cycle ;
\draw  [color={rgb, 255:red, 74; green, 144; blue, 226 }  ,draw opacity=1 ][fill={rgb, 255:red, 74; green, 144; blue, 226 }  ,fill opacity=0.16 ] (137.84,145.03) -- (507.83,145.03) -- (507.83,153.02) -- (137.84,153.02) -- cycle ;
\draw  [color={rgb, 255:red, 66; green, 118; blue, 179 }  ,draw opacity=1 ][fill={rgb, 255:red, 74; green, 144; blue, 226 }  ,fill opacity=0.16 ] (269.06,182.73) -- (507.45,182.73) -- (507.45,190.71) -- (269.06,190.71) -- cycle ;
\draw  [color={rgb, 255:red, 74; green, 144; blue, 226 }  ,draw opacity=1 ][pattern=_perv8pg8o,pattern size=6pt,pattern thickness=0.75pt,pattern radius=0pt, pattern color={rgb, 255:red, 74; green, 144; blue, 226}] (181.8,118.1) -- (507.83,118.1) -- (507.83,126.08) -- (181.8,126.08) -- cycle ;
\draw  [color={rgb, 255:red, 74; green, 144; blue, 226 }  ,draw opacity=1 ][pattern=_tm4jdcxlk,pattern size=6pt,pattern thickness=0.75pt,pattern radius=0pt, pattern color={rgb, 255:red, 74; green, 144; blue, 226}] (312.49,52.85) -- (508.82,52.85) -- (508.82,60.83) -- (312.49,60.83) -- cycle ;
\draw  [color={rgb, 255:red, 74; green, 144; blue, 226 }  ,draw opacity=1 ][pattern=_nlv3f6ai7,pattern size=6pt,pattern thickness=0.75pt,pattern radius=0pt, pattern color={rgb, 255:red, 74; green, 144; blue, 226}] (269.06,182.73) -- (507.45,182.73) -- (507.45,190.71) -- (269.06,190.71) -- cycle ;
\draw [color={rgb, 255:red, 208; green, 2; blue, 27 }  ,draw opacity=1 ] [dash pattern={on 4.5pt off 4.5pt}]  (368.87,53.96) -- (368.87,163.56) -- (368.87,233.39) ;
\draw  [color={rgb, 255:red, 208; green, 2; blue, 27 }  ,draw opacity=1 ][fill={rgb, 255:red, 208; green, 2; blue, 27 }  ,fill opacity=1 ] (366.15,161.38) .. controls (366.15,159.88) and (367.37,158.67) .. (368.87,158.67) .. controls (370.37,158.67) and (371.59,159.88) .. (371.59,161.38) .. controls (371.59,162.89) and (370.37,164.1) .. (368.87,164.1) .. controls (367.37,164.1) and (366.15,162.89) .. (366.15,161.38) -- cycle ;
\draw [color={rgb, 255:red, 139; green, 87; blue, 42 }  ,draw opacity=1 ] [dash pattern={on 4.5pt off 4.5pt}]  (224.21,149.42) -- (224.21,227.04) ;
\draw [color={rgb, 255:red, 139; green, 87; blue, 42 }  ,draw opacity=1 ] [dash pattern={on 4.5pt off 4.5pt}]  (224.21,149.42) -- (100,149.42) ;
\draw  [color={rgb, 255:red, 0; green, 0; blue, 0 }  ,draw opacity=1 ][fill={rgb, 255:red, 0; green, 0; blue, 0 }  ,fill opacity=1 ] (221.5,149.42) .. controls (221.5,147.92) and (222.71,146.7) .. (224.21,146.7) .. controls (225.72,146.7) and (226.93,147.92) .. (226.93,149.42) .. controls (226.93,150.92) and (225.72,152.14) .. (224.21,152.14) .. controls (222.71,152.14) and (221.5,150.92) .. (221.5,149.42) -- cycle ;
\draw [color={rgb, 255:red, 139; green, 87; blue, 42 }  ,draw opacity=1 ]   (224.21,221.04) -- (224.21,233.01) ;
\draw [color={rgb, 255:red, 139; green, 87; blue, 42 }  ,draw opacity=1 ]   (104,149.42) -- (92.58,149.42) ;
\draw [color={rgb, 255:red, 74; green, 144; blue, 226 }  ,draw opacity=1 ]   (269.06,190.75) -- (534.33,190.75) ;

\draw (525.83,227.06) node [anchor=north west][inner sep=0.75pt]    {$t$};
\draw (66.67,215.1) node [anchor=north west][inner sep=0.75pt]  [color={rgb, 255:red, 0; green, 0; blue, 0 }  ,opacity=1 ]  {$0$};
\draw (66.67,41.08) node [anchor=north west][inner sep=0.75pt]  [color={rgb, 255:red, 0; green, 0; blue, 0 }  ,opacity=1 ]  {$1$};
\draw (125.79,232.63) node [anchor=north west][inner sep=0.75pt]  [font=\normalsize,color={rgb, 255:red, 139; green, 87; blue, 42 }  ,opacity=1 ]  {$S_{1}$};
\draw (171.3,232.63) node [anchor=north west][inner sep=0.75pt]  [font=\normalsize,color={rgb, 255:red, 139; green, 87; blue, 42 }  ,opacity=1 ]  {$S_{2}$};
\draw (53.96,161.76) node [anchor=north west][inner sep=0.75pt]  [font=\small,color={rgb, 255:red, 139; green, 87; blue, 42 }  ,opacity=1 ]  {$U_{1}$};
\draw (53.96,120.52) node [anchor=north west][inner sep=0.75pt]  [font=\small,color={rgb, 255:red, 139; green, 87; blue, 42 }  ,opacity=1 ]  {$U_{2}$};
\draw (54.1,78.01) node [anchor=north west][inner sep=0.75pt]  [font=\small,color={rgb, 255:red, 139; green, 87; blue, 42 }  ,opacity=1 ]  {$U_{N}$};
\draw (378.13,232.63) node [anchor=north west][inner sep=0.75pt]  [font=\normalsize,color={rgb, 255:red, 139; green, 87; blue, 42 }  ,opacity=1 ]  {$\ \ \ \ \ \ \ \ \cdots \ \ \ \ \ \ \ \ S_{N}$};
\draw (516.99,51.91) node [anchor=north west][inner sep=0.75pt]  [font=\normalsize,color={rgb, 255:red, 74; green, 144; blue, 226 }  ,opacity=1 ]  {$I_{6}$};
\draw (516.99,74.75) node [anchor=north west][inner sep=0.75pt]  [font=\normalsize,color={rgb, 255:red, 74; green, 144; blue, 226 }  ,opacity=1 ]  {$I_{5}$};
\draw (515.9,96.5) node [anchor=north west][inner sep=0.75pt]  [font=\normalsize,color={rgb, 255:red, 74; green, 144; blue, 226 }  ,opacity=1 ]  {$I_{4}$};
\draw (515.9,120.43) node [anchor=north west][inner sep=0.75pt]  [font=\normalsize,color={rgb, 255:red, 74; green, 144; blue, 226 }  ,opacity=1 ]  {$I_{3}$};
\draw (515.9,151.97) node [anchor=north west][inner sep=0.75pt]  [font=\normalsize,color={rgb, 255:red, 74; green, 144; blue, 226 }  ,opacity=1 ]  {$I_{2}$};
\draw (515.9,193.3) node [anchor=north west][inner sep=0.75pt]  [font=\normalsize,color={rgb, 255:red, 74; green, 144; blue, 226 }  ,opacity=1 ]  {$I_{1}$};
\draw (258.31,232.63) node [anchor=north west][inner sep=0.75pt]  [font=\normalsize,color={rgb, 255:red, 139; green, 87; blue, 42 }  ,opacity=1 ]  {$S_{4}
\ \ \ \ \ \ \cdots$};
\draw (79.31,176.64) node [anchor=north west][inner sep=0.75pt]  [font=\scriptsize,color={rgb, 255:red, 74; green, 144; blue, 226 }  ,opacity=1 ]  {$l_{2}$};
\draw (79.31,139.66) node [anchor=north west][inner sep=0.75pt]  [font=\scriptsize,color={rgb, 255:red, 74; green, 144; blue, 226 }  ,opacity=1 ]  {$l_{3}$};
\draw (79.31,112.83) node [anchor=north west][inner sep=0.75pt]  [font=\scriptsize,color={rgb, 255:red, 74; green, 144; blue, 226 }  ,opacity=1 ]  {$l_{4}$};
\draw (78.95,89.63) node [anchor=north west][inner sep=0.75pt]  [font=\scriptsize,color={rgb, 255:red, 74; green, 144; blue, 226 }  ,opacity=1 ]  {$l_{5}$};
\draw (80.04,68.96) node [anchor=north west][inner sep=0.75pt]  [font=\scriptsize,color={rgb, 255:red, 74; green, 144; blue, 226 }  ,opacity=1 ]  {$l_{6}$};
\draw (361.05,232.59) node [anchor=north west][inner sep=0.75pt]  [font=\normalsize,color={rgb, 255:red, 208; green, 2; blue, 27 }  ,opacity=1 ]  {$S_{\tau }$};
\draw (213.31,232.63) node [anchor=north west][inner sep=0.75pt]  [font=\normalsize,color={rgb, 255:red, 139; green, 87; blue, 42 }  ,opacity=1 ]  {$S_{3}$};
\draw (53.96,137.76) node [anchor=north west][inner sep=0.75pt]  [font=\small,color={rgb, 255:red, 139; green, 87; blue, 42 }  ,opacity=1 ]  {$U_{3}$};
\draw (535.98,181.64) node [anchor=north west][inner sep=0.75pt]  [font=\scriptsize,color={rgb, 255:red, 74; green, 144; blue, 226 }  ,opacity=1 ]  {$r_{1} =l_{2} -1/( n-1)$};

\end{tikzpicture}
    \caption{
    The black dots represent the atoms $((S_i,U_i),i \ge 1)$ of the Poisson process $\mathbf{N}$. 
    The union of the striped blue regions is the ``forbidden'' region up to the stopping time $S_\tau$; the projection of the striped blue regions onto the $y$-axis is $C_{\tau-1}$.}
    \label{fig:poisson}
\end{figure}
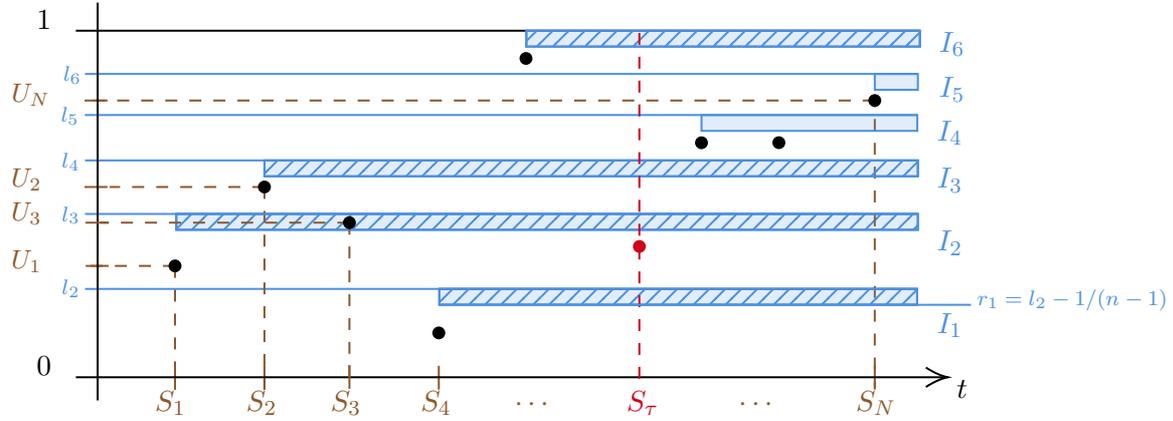

Fix any vector $(j(1),\ldots,j(\ell))$ of distinct elements of $[n]$ and any increasing sequence of  integers $(m(1),\ldots,m(\ell-1))$ with $m(1)=1$. Suppose that $\tau > M(\ell-1)$, that $M(k)=m(k)$ and that $J(m(k))=j(k)$ for all $1 \le k < \ell$. Then $\tau \le M(\ell)$ (and so $\tau < M(\ell)$) precisely if the first point among $(U_m,m > m(\ell-1))$ which does {\em not} fall into $C_{m(\ell-1)} = \bigcup_{k=1}^{\ell-1} [r_{j(k)},l_{j(k)+1})$, {\em does} belong to the set
\[
\bigcup_{k=1}^{\ell-1} I_{j(k)} \setminus C_{m(\ell-1)}
=
\bigcup_{k=1}^{\ell-1} [l_{j(k)},r_{j(k)}). 
\]
Since $r_{j(k)}-\ell_{j(k)} = (d_{j(k)}-1)/(n-1)$, writing $U$ for a Uniform$[0,1]$ random variable, it follows that 
\begin{align*}
& \probC{\tau < M(\ell)}{\tau > M(\ell-1),M(k)=m(k)\mbox{ and }J(m(k))=j(k)\mbox{ for all }1 \le k < \ell}\\
& = \probC{U \in \bigcup_{k=1}^{\ell-1} [l_{j(k)},r_{j(k)})}{U \not\in \bigcup_{k=1}^{\ell-1} [r_{j(k)},l_{j(k)+1})}  \\
& = 
\frac{\sum_{k=1}^{\ell-1} (d_{j(k)}-1)/(n-1)}{1-((\ell-1)/(n-1))} = \frac{\sum_{k=1}^{\ell -1}(d_{j(k)}-1)}{n-\ell}\, .
\end{align*}
Since the final expression depends on $(j(1),\ldots,j(\ell-1))$ and $(m(1),\ldots,m(\ell-1))$ only through the values $(d_{j(1)},\ldots,d_{j(\ell-1)})$, we have  
\begin{align}\label{eq:tau_condprob}
& \probC{\tau \le M(\ell)}{D(1),\ldots,D(\ell-1),\tau > M(\ell-1)} \nonumber\\
& = \probC{\tau < M(\ell)}{D(1),\ldots,D(\ell-1),\tau > M(\ell-1)} \nonumber\\
& = \frac{\sum_{k=1}^{\ell-1} (D(k)-1)}{n-\ell}. 
\end{align}
Writing $\sigma=\sup(k\ge 1: \tau > M(k))$, the preceding identity implies that $\sigma$ has the same distribution as the random variable from  Theorem~\ref{thm:main_mod} with the same name. Since $M(k) \ge k$, it follows that $\sigma \le \sup(k \ge 1: \tau > k) = \tau-1$, so for all $\ell \ge 1$, 
\begin{equation}\label{eq:sigma_tau_ident}
\p{\sigma \ge \ell} = \p{\tau > M(\ell)} \le \p{\tau > \ell}\, .
\end{equation}
The second bound of Theorem~\ref{thm:main_mod} now follows quite straightforwardly: if $\dstat(1)=0$ then $[l_i,r_{i+1}) = (d_i-1)/(n-1) \ge 1/(n-1)$ whenever $d_i \ne 0$. Therefore, for any $\ell \ge 2$, if $\tau > M(\ell-1)$ then $D(k) \ge 2$ for all $1 \le k \le \ell-1$, so by (\ref{eq:tau_condprob}), 
\[
\probC{\tau \le M(\ell)}{\tau > M(\ell-1)}  \ge \frac{\ell-1}{n-\ell}.
\]
It follows by induction that 
\[
\p{\sigma\ge \ell}=\p{\tau > M(\ell)} \le \prod_{k=1}^{\ell-1} 
\pran{1-\frac{k}{n-k}} \le e^{-(\ell-1)^2/(2(n-1))}\, ,
\]
proving the second inequality in Theorem~\ref{thm:main_mod}. 
We now turn to proving the first inequality in Theorem~\ref{thm:main_mod}; this bound is an immediate consequence of the next proposition. 
\begin{prop}\label{prop:main_mod}
Write $\varmod=\sum_{i:d_i \ge 2} d_i^2/(n-1)$. 
Then for all $\beta \ge 17^{3/2}$, 
\[
\p{\tau > \beta \pran{\frac{n-1}{\varmod}}^{1/2}}
\le 
\exp\pran{-\frac{1}{3}\pran{\frac{\beta^{2/3}(n-1)}{\varmod}}^{1/2}} + 2\exp\pran{-\frac{\beta^{2/3}}{24}}\, .
\]
\end{prop}
The first bound of Theorem~\ref{thm:main_mod} follows from the proposition since $\sum_{i:d_i \ge 2} d_i^2 = |\dstat|_2^2 - \dstat(1)$ and $|\dstat|_1=(n-1)$, so $\tfrac{n-1}{\varmod} = \tfrac{|\dstat|_1^2}{|\dstat|_2^2 - \dstat(1)}$. 

The proposition's proof is where the Poisson process setup comes into its own. Write $\bN(t) = \bN([0,t]\times[0,1])$ for the number of points of $\bN$ arriving by time $t$, and let $\bN_i(t) = \bN([0,t]\times[l_i,r_i))$ be the number of points arriving in the interval $[l_i,r_i)$ by time $t$. 
Note that for all $i \in [n]$, if $\bN_i(t) \ge 2$ then $\tau \le \bN(t)$. Thus, letting $T =\inf\{t \ge 0: \max_{i \in [n]}\bN_i(t) \ge 2\}$, we have $\tau \le \bN(T)$. It follows that for all $h \in \N$, if $\bN(t) \le h$ and $T \le t$ then $\tau \le h$, so 
 \begin{equation}\label{eq:generic_poissonization_bound}
 \p{\tau > h} \le \inf_{t \ge 0} \Big(\p{\bN(t) > h} + \p{T > t}\Big)\, .
 \end{equation}
We control the first of these probabilities using standard Poisson tail estimates. The second requires a little more work. The random variables $(\bN_i(t),i \in [n])$ are independent and $\bN_i(t)$ is Poisson$(t(r_i-l_i))$-distributed. 
Note that $r_i-l_i=0$ when $d_i \le 1$. 
Writing $p_i = d_i/(2(n-1))$, we have $r_i-l_i \ge p_i$ whenever $d_i \ge 2$, so
\begin{equation}\label{eq:poisson_upper}
\p{T>t} 
= \prod_{i:d_i \ge 2} \p{\bN_i(t) \le 1} 
\le\prod_{i:d_i \ge 2} \p{\mathrm{Poisson}(p_it)\le 1} 
 = \prod_{i: d_i \ge 2} (1+p_it)e^{-p_it}\, .
\end{equation}
Combining this with (\ref{eq:generic_poissonization_bound}) and the tail bound 
$\p{\mathrm{Poisson}(t)> h} 
=e^{-t((h/t)\log(h/t) - h/t + 1)}$, which holds for $h \ge t$ and can be found in, e.g., \cite[Page 23]{MR3185193}, we obtain that 
\begin{equation}\label{eq:poissonization_bd}
\p{\tau > h}
\le \inf_{t \le h}
\pran{
e^{-t((h/t)\log(h/t)- h/t + 1)} + \prod_{i: d_i \ge 2} (1+p_it)e^{-p_it}\, 
}.
\end{equation}
We next focus on proving bounds for the second term on the right-hand side of (\ref{eq:poissonization_bd}); our approach is based on that of Lemma~9 in \cite{MR1741774}. 
\begin{lem}\label{lem:gtd_bd}
Write $\bd=(d_1,\ldots,d_n)$ and 
let $g(t,\bd) = \prod_{i: d_i \ge 2} (1+p_it)e^{-p_i t}$, where $p_i = d_i/(2(n-1))$. Also write $p_{\max} = \max_{i \in [n]} p_i$ and $d_{\max}=\max_{i \in [n]} d_i$. Then 
for all $0 \le t < 1/p_{\max}=2(n-1)/d_{\max}$, 
\[
\log g(t,\bd) = \sum_{k \ge 2} \frac{(-1)^{k+1}}{k}\sum_{i: d_i \ge 2}\pran{\frac{d_it}{2(n-1)}}^k\, ,
\]
and with $\varmod=\sum_{i:d_i \ge 2} d_i^2/(n-1)$, we have 
\[
\left|
\log g(t,\bd)+\frac{\varmod t^2}{8(n-1)}
\right|
\le \frac{d_{\max} t}{6(n-1)-3d_{\max}t}\cdot \frac{\varmod t^2}{4(n-1)} \, .
\]
\end{lem}
\begin{proof}
First, note that 
\[
\sum_{k \ge 2} \frac{1}{k} \sum_{i:d_i \ge 2} (p_i t)^k < \infty
\]
for $0 \le t < 1/p_{\max}$, since the inner sum has finitely many summands, each of which decreases geometrically in $k$. By a Taylor expansion of $\log(1+x)$ around $x=0$ and Tonelli's theorem, it follows that  
\begin{align*}
\log g(t,\bd) 
& = \sum_{i: d_i \ge 2} (\log(1+p_i t)-p_i t)\\
& = \sum_{i: d_i \ge 2} \sum_{k\ge 1} \frac{(-1)^{k+1}}{k} (p_it)^k - 
\sum_{i: d_i \ge 2}p_i t\\
& = \sum_{k \ge 2} \frac{(-1)^{k+1}}{k}\sum_{i: d_i \ge 2} (p_it)^k\, .
\end{align*}
Next, note that 
\[
\sum_{i: d_i \ge 2} p_i^k \le p_{\max}^{k-2}\sum_{i:d_i \ge 2} p_i^2,
\]
so 
\[
\sum_{k \ge 3}\sum_{i: d_i \ge 2} \frac{(p_it)^k}{k} 
\le \pran{\sum_{i:d_i \ge 2} (p_it)^2}
\sum_{k \ge 3} \frac{(p_{\max}t)^{k-2}}{k}
\le \pran{\sum_{i:d_i \ge 2} (p_it)^2}
\cdot \frac{p_{\max} t}{3(1-p_{\max}t)}\, 
\]
and thus 
\[
\left|\log g(t,\bd) + \sum_{i:d_i \ge 2} \frac{(p_it)^2}{2}\right| \le 
\pran{\sum_{i:d_i \ge 2} (p_it)^2}
\cdot \frac{p_{\max} t}{3(1-p_{\max}t)}\, .
\]
Using that $p_i = d_i/(2(n-1))$ and $p_{\max}=d_{\max}/(2(n-1))$ and $\varmod=\sum_{i:d_i \ge 2} d_i^2/(n-1)$, this is precisely the bound claimed in the lemma; this completes the proof. 
\end{proof}
\begin{cor}\label{cor:gtd_bd}
Write $\varmod=\sum_{i:d_i \ge 2} d_i^2/(n-1)$. Then for all $0 \le t \le (n-1)/d_{\max}$, 
\[
g(t,\bd) \le \exp\pran{\frac{-\varmod t^2}{24(n-1)}}.
\]
\end{cor}
\begin{proof}
For $t \le (n-1)/d_{\max}$ we have 
\[
\frac{d_{\max} t}{6(n-1)-3d_{\max}t}
\cdot 
\frac{\varmod t^2}{4(n-1)}
\le \frac{\varmod t^2}{12(n-1)},
\]
and the bound in the lemma then gives 
\[
\log g(t,\bd) 
\le
-\frac{\varmod t^2}{8(n-1)}
+ \frac{\varmod t^2}{12(n-1)} 
=-\frac{\varmod t^2}{24(n-1)}\, .\qedhere
\]
\end{proof}
\begin{proof}[Proof of Proposition~\ref{prop:main_mod}]
First, if $d_{\max}=1$ then $\varmod=0$ and the lemma asserts a non-negative upper bound on $\p{\tau > \infty}$, so clearly holds. We thus assume that $d_{\max} > 1$ for the rest of the proof.

Fix $C > 2$. If $d_{\max} \le (\sum_{i:d_i \ge 2} d_i^2)^{1/2}/C=\tfrac{((n-1)\varmod)^{1/2}}{C}$, then taking $t = \tfrac{C(n-1)^{1/2}}{\varmod^{1/2}} \le \tfrac{n-1}{d_{\max}}$, by 
Lemma~\ref{lem:gtd_bd} we have
\begin{align*}
g(t,\bd) \le \exp\pran{-\frac{\varmod t^2}{24(n-1)}} = \exp\pran{-\frac{C^2}{24}}\, .
\end{align*}
Taking $h=2t=\tfrac{2C(n-1)^{1/2}}{\varmod^{1/2}}$, and noting that 
\[
t\pran{\frac{h}{t}\log\frac{h}{t}-\frac{h}{t} + 1}
= t(2\log 2 - 1)> \frac{t}{3}\, ,
\]
by (\ref{eq:poissonization_bd}) we obtain that 
\begin{align}\label{eq:dsmall_taubd}
\p{\tau > 2C \frac{(n-1)^{1/2}}{\varmod^{1/2}}}
& \le \exp\pran{-\frac{C}{3}\frac{(n-1)^{1/2}}{\varmod^{1/2}}} + \exp\pran{-\frac{C^2}{24}}\, .
\end{align}

Now suppose that $d_{\max} > (\sum_{i:d_i \ge 2} d_i^2)^{1/2}/C=\tfrac{((n-1)\varmod)^{1/2}}{C}$. By construction the entries of $(d_1,\ldots,d_n)$ are non-decreasing, so $d_n=d_{\max}$. For any positive real $K\ge 2$, if at least two of the points $U_1,\ldots,U_{\lfloor K\rfloor}$ lie in the interval $[l_n,r_n)$ then $\tau \le K$, so 
\begin{align*}
\p{\tau > K} 
&\le \p{|\{k \in [\lfloor K \rfloor]: U_k \in [l_n,r_n)\}|\le 1}\\
& = \p{\mathrm{Bin}(\lfloor K \rfloor,r_n-l_n) \in \{0,1\}}\\
& = (1-(r_n-l_n))^{\lfloor K \rfloor-1}(1-(r_n-l_n)+\lfloor K \rfloor(r_n-l_n))\, .
\end{align*}
Since $d_{\max} \ge 2$ we have $r_n-l_n = (d_{\max}-1)/(n-1) \ge d_{\max}/(2(n-1))$, so using that $1-x \le e^{-x}$ it follows that for $K \ge 2C(\tfrac{n-1}{\varmod})^{1/2}\ge 4$, 
\begin{align}
\p{\tau > K} 
& \le \pran{1+\frac{(\lfloor K \rfloor-1)d_{\max}}{n-1}}\cdot\exp\pran{-\frac{(\lfloor K \rfloor-1) d_{\max}}{2(n-1)}}\, \nonumber\\
& \le \frac{2Kd_{\max}}{n-1}\exp\pran{-\frac{Kd_{\max}}{4(n-1)}}\, ,\label{eq:tau_bd_d_big}
\end{align}
where for the second inequality we have used the lower bound on $K$ to deduce that 
$\lfloor K\rfloor -1 > K/2$ and that 
$\tfrac{(\lfloor K \rfloor-1)d_{\max}}{n-1} > \tfrac{K}{2}\tfrac{d_{\max}}{n-1} \ge 2$ and so $1 + \tfrac{(\lfloor K \rfloor-1)d_{\max}}{n-1} < \tfrac{2Kd_{\max}}{n-1}$. 

Taking $K=xC(\tfrac{n-1}{\varmod})^{1/2}$ for $x \ge 4$, the lower bound on $d_{\max}$ implies that $Kd_{\max}/(n-1) \ge x$; since $2xe^{-x/4}$ is decreasing for $x \ge 4$, the bound (\ref{eq:tau_bd_d_big}) then implies that 
\begin{align}\label{eq:dbig_taubd}
\p{\tau \ge xC\frac{(n-1)^{1/2}}{\varmod^{1/2}}}
\le 2xe^{-x/4}\, .
\end{align}

To finish the proof, we combine (\ref{eq:dsmall_taubd}) and (\ref{eq:dbig_taubd}) to get a bound which does not depend on the value of $d_{\max}$. Take $\beta \ge 17^{3/2}$, let $C=\beta^{1/3}>2$ and $x=\beta^{2/3}\ge 4$. Then $2C \le \beta$ and $xC=\beta$. Whatever the value of $d_{\max}$, one of (\ref{eq:dsmall_taubd}) and (\ref{eq:dbig_taubd}) applies, so we obtain that 
\begin{align*}
\p{\tau \ge \beta\frac{(n-1)^{1/2}}{\varmod^{1/2}}}
& = \p{\tau \ge xC\frac{(n-1)^{1/2}}{\varmod^{1/2}}}\\
& \le 
\exp\pran{-\frac{C}{3}\frac{(n-1)^{1/2}}{\varmod^{1/2}}} + \exp\pran{-\frac{C^2}{24}}
+ 2x e^{-x/4}\\
& = 
\exp\pran{-\frac{\beta^{1/3}}{3}\frac{(n-1)^{1/2}}{\varmod^{1/2}}}
+
\exp\pran{-\frac{\beta^{2/3}}{24}}
+
2\beta^{2/3}\exp\pran{-\frac{\beta^{2/3}}{4}}\, .
\end{align*}
Finally, it is straightforward to check that $e^{-y/24}+2ye^{-y/4} \le 2e^{-y/24}$ for $y \ge 17$, which combined with the previous inequality yields the first bound of the proposition.
\end{proof}

\section{\bf Proofs of the conjectures from \cite{MR2908619} and of Propositions~
\ref{prop:dstats_infinite_variance},~\ref{prop:dstats_zero_radius} and~\ref{prop:dstats_finite_variance}}
\label{sec:simply_generated}

The sort of random trees considered by Janson~\cite{MR2908619} are called {\em simply generated trees}; they are defined as follows. Fix non-negative real weights $\wts=(w_k,k \ge 0)$ with $w_0>0$. For a finite plane tree $t$, the weight of $t$ is 
\[
\wts(t) = \prod_{v \in \vset(t)} w_{\deg_t(v)}\, .
\]
For positive integers $n$ write 
\[
Z_n=Z_n(\wts) = \sum_{\mathrm{plane~trees}~t: |\vset(t)|=n} \wts(t) \, ,
\]
and when $Z_n > 0$ define a random tree $\sgt_n=\sgt_n(\wts)$ by 
\[
\p{\sgt_n=t} = \frac{\wts(t)}{Z_n}\,
\]
for plane trees $t$ with $|\vset(t)|=n$. Then $\sgt_n$ is called a {\em simply generated tree of size $n$ with weight sequence $\wts$}. 
If $\sum_{k \ge 0} w_k=1$ then 
$\cT_n(\wts)$ is distributed as a Bienaym\'e tree with offspring distribution $\wts$ conditioned to have $n$ vertices. 

Write $\Phi(z) = \Phi_{\wts}(z) = \sum_{k \ge 0} w_k z^k$ for the generating function of the sequence $\wts$, and $\rho=\rho_{\wts}$ for the radius of convergence of $\Phi$.  For $t > 0$ such that $\Phi(t) < \infty$, define 
\[
\Psi(t) = \Psi_{\wts}(t)=\frac{t\Phi'(t)}{\Phi(t)} = \frac{\sum_{k\ge 0} kw_kt^k}{\sum_{k \ge 0} w_kt^k}\, .
\]
If $\Phi(\rho)=\infty$ then also define 
\[
\Psi(\rho)=\Psi_\wts(\rho) = \lim_{t \uparrow \rho} \Psi(t)\, ;
\]
the function $\Psi$ is strictly increasing on $[0,\rho)$ by \cite[Lemma~3.1(i)]{MR2908619}, so this limit exists. In all cases, write $\nu=\nu(\wts)=\Psi_\wts(\rho)$. Note that $\Psi(t) \in (0,\infty]$ for all $t > 0$, so $\nu=0$ if and only if $\rho=0$.

The questions from \cite{MR2908619} that we address in this paper concern exclusively weight sequences with $\nu \le 1$, and we assume this is the case from now on. We define $\sigma^2 = \rho\Psi'(\rho)$; this is a slight simplification of the definition from \cite[Theorem~7.1]{MR2908619}, made possible by the assumption that $\nu \le 1$. 

The following conjecture summarizes Conjectures~21.5 and~21.6 and Problems~21.7 and~21.8 from \cite{MR2908619}.
\begin{conj}[\cite{MR2908619}]\label{conj_svante}
Let $\wts=(w_k,k \ge 0)$ be a weight sequence with $w_0>0$ and with $w_k>0$ for some $k \ge 2$, and for $n \ge 0$ with $Z_n(\wts)>0$ let $\sgt_n$ be a simply generated tree of size $n$ with weight sequence $\wts$.
\begin{enumerate}
\item If $\nu=1$ and $\sigma^2 = \infty$ then $\hyt(\sgt_n)/n^{1/2} \to 0$  in probability. 
\item If $\nu=1$ and $\sigma^2 = \infty$ then $\wid(\sgt_n)/n^{1/2} \to \infty$ in probability. 
\item If $\nu < 1$ then $\hyt(\sgt_n)/n^{1/2} \to 0$ in probability. 
\item If $\nu < 1$ then $\wid(\sgt_n)/n^{1/2} \to \infty$ in probability.
\end{enumerate}
In all four statements, the convergence is as $n \to \infty$ along integers $n$ such that $Z_n(\wts)>0$. 
\end{conj}

The results of this work establish points (2) and (4) of this conjecture, and establish (1) and (3) up to polylogarithmic factors. To make these deductions, we rely on the following result from \cite{MR2908619} about the typical degree statistics of simply generated trees. 
\begin{thm}[\cite{MR2908619}]\label{thm:dstats_simply_generated}
Let $\wts=(w_k,k \ge 0)$ be a weight sequence with $w_0>0$ and with $w_k>0$ for some $k \ge 2$. Whenever $Z_n(\wts)>0$ let $\sgt_n$ be a simply generated tree with weight sequence $\wts$ and size $n$. 
Assume that $\nu=\nu(\wts)\in (0,1]$. 
Writing $\rho=\rho_{\wts}\in (0,\infty]$, for $k \ge 0$ let 
\begin{equation}\label{eq:pi_def}
\pi(k) := \frac{w_k \rho^k}{\Phi(\rho)}\, .
\end{equation}
Then $\pi=(\pi(k),k \ge 0)$ is a probability distribution, with expectation $\nu$ and variance $\sigma^2 = \rho\Psi'(\rho)$, and the degree statistics $\dstat_{\sgt_n}$ satisfy that for every integer $k \ge 0$ and real $\eps > 0$, 
\[
\p{\left|\frac{\dstat_{\sgt_n}(k)}{n} - \pi(k)\right| > \eps } = \oexp(1)\, .
\]
\end{thm}
This theorem is essentially a special case of \cite[Theorem 11.4]{MR2908619}. The error bounds stated above are not made explicit in the statement of that theorem, but are recorded in the course of its proof (see \cite[page 164]{MR2908619}). 

We also require a version of Theorem~\ref{thm:dstats_simply_generated} which addresses the case that $\nu=\rho=0$. 
Before stating this result, note that if $\rho_{\wts}=0$ then the probability distribution $\pi$ defined by \eqref{eq:pi_def} has
$\pi(0)=1$ and $\pi(k)=0$ for $k > 0$.
\begin{thm}\label{thm:nu_zero}
Let $\wts=(w_k,k \ge 0)$ be a weight sequence with $w_0>0$ and with $w_k>0$ for some $k \ge 2$. Whenever $Z_n(\wts)>0$ let $\sgt_n$ be a simply generated tree with weight sequence $\wts$ and size $n$. 
Suppose that $\rho_\wts=0$. Then the degree statistics $\dstat_{\sgt_n}$ satisfy that for every real $\eps > 0$, 
\[
\p{\frac{\dstat_{\sgt_n}(0)}{n} < 1- \eps} = \oexp(1)\, .
\]
\end{thm}
This theorem asserts that when the radius of convergence of $\Phi$ is zero, with very high probability $\sgt_n$ has $n-o(n)$ leaves.
\begin{proof}
Fix $\delta > 0$ and integer $L>2$. We claim that 
\begin{equation}\label{eq:toprove}
\frac{\p{\sum_{1 \le c \le L} \dstat_{\sgt_n}(c) \ge 2\delta n}}{\p{\sum_{1 \le c \le L} \dstat_{\sgt_n}(c) \le \delta n}} = \oexp(1)\, .
\end{equation}
Since $\sum_{c \ge 1} c\dstat_{\sgt_n}(c) =n-1$, 
we deterministically have $\sum_{c > L} \dstat_{\sgt_n}(c) \le (n-1)/L$,
so the bound \eqref{eq:toprove} implies that 
\[
\p{\sum_{c \ge 1} \dstat_{\sgt_n}(c) \le 2\delta n + \frac{n-1}{L}} \ge
\p{\sum_{1 \le c \le L} \dstat_{\sgt_n}(c) \le 2\delta n} = 1-\oexp(1).
\]
Since $\delta > 0$ and $L>2$ are arbitrary, this proves the theorem. It thus remains to prove~\eqref{eq:toprove}.

It's convenient in what follows to assume that $w_0=1$. (We can achieve this by multiplying all weights by $w_0^{-1}$; this does not change the distribution of $\sgt_n$.) 
Now fix $K$ large enough that $K^\delta > 2(L+1)$.
Note that if $\limsup_{k \to \infty}(\log w_k)/k = r < \infty$, then $\rho_\wts \ge 1/r$; since we assume $\rho_\wts=0$, it follows that $\limsup_{k \to \infty}(\log w_k)/k=\infty$, 
so we may further choose an integer $M > 2L$ such that 
\[
w_M \ge 
\max
\pran{(K w_c)^{M/c}, 0 < c \le L}. 
\]

In what follows, given a sequence $\dstat$ which is the degree statistics of a tree (so $\sum_{c \ge 0}\dstat(c)=|\dstat|_1+1$), it's useful to write 
$\wts(\dstat) := \prod_{c \ge 0} w_c^{\dstat(c)}$ --- so if $t$ is a tree with degree statistics $\dstat$ then $\wts(t)=\wts(\dstat)$. 
Now, for such a sequence $\dstat$, form degree statistics $\hat{\dstat}$ as follows. 
For $c \in \N$ let $m(c) = \lfloor \dstat(c)/M\rfloor$. Then for $i \in \N$ define 
\[
\hat{\dstat}(i) = 
\begin{cases}
\dstat(0) + \sum_{0 < c \le L} (M-c)m(c)	& \mbox{ if } i=0\\
\dstat(i)-Mm(i) &\mbox{ if }0 < i \le L \\
\dstat(i)+\sum_{0 < c \le L} cm(c) &\mbox{ if }i=M\\
\dstat(i) &\mbox{ otherwise.}
\end{cases}
\]
Then $|\hat{\dstat}|_1=|\dstat|_1$ and $\sum_{c \ge 0} c\hat{\dstat}(c) = \sum_{c \ge 0} \dstat(c)$, so $\hat{\dstat}$ is again the degree statistics of a tree with $|\dstat|_1+1$ vertices. 
Since $w_0=1$, we also have 
\begin{align*}
\frac{\wts(\hat{\dstat})}{\wts(\dstat)}
& = 
(w_M)^{\sum_{0 < c \le L} cm(c)} 
\cdot
\prod_{0 < c \le L} (w_c)^{-Mm(c)}
 = 
\prod_{0 < c \le L}
\pran{\frac{w_M^c}{w_c^M}}^{m(c)}
\end{align*}
Since $w_M \ge (Kw_c)^{M/c}$, this yields that 
\[
\frac{\wts(\hat{\dstat})}{\wts(\dstat)}
\ge \prod_{0 < c\le L}
K^{Mm(c)} = K^{\sum_{0 < c \le L} M\lfloor \dstat(c)/M\rfloor}\, .
\]

Now write $n = \sum_{c \ge 0} \dstat(c)=\sum_{c \ge 0} \hat{\dstat}(c)$. 
Note that for all $i\in\N$ we have $\dstat(i)-Mm(i) \le M-1$, so 
 $\sum_{0 < c \le L} \hat{\dstat}(c) \le (M-1)L$, and thus if $n \ge (M-1)L/\delta$ then $\sum_{0 < c \le L} \hat{\dstat}(c) \le \delta n$.
For such $n$, if $\sum_{0 < c \le L} \dstat(c) \ge 2\delta n$, then we also have 
$\sum_{0 < c \le L} M \lfloor \dstat(c)/M\rfloor \ge (\sum_{0 < c \le L} \dstat(c))-(M-1)L \ge \delta n$, so it follows from the previous lower bound on $\wts(\hat{\dstat})/\wts(\dstat)$ that 
\begin{equation}\label{eq:weight_gain_is_good}
\frac{\wts(\hat{\dstat})}{\wts(\dstat)} \ge K^{\delta n}\, .
\end{equation}

To use this bound to complete the proof, it remains to control (a) the number of degree statistics $\dstat$ that can give rise to a given degree statistics $\hat{\dstat}$, and (b), for a given pair of degree statistics $\dstat$ and $\hat{\dstat}$, the relative numbers of trees with these degree statistics. 

To control (a), 
fix a sequence $\dstat'$ which is the degree statistics of a tree of size $n$. Then for any degree statistics $\dstat$ with $\hat{\dstat}=\dstat'$, there are non-negative integers $m_1,\ldots,m_L$ with $\sum_{0 < c \le L} cm_c < n$
such that $\dstat'(c)=\dstat(c)-Mm_c$ for $0 < c \le L$. Moreover, $\dstat$ may be reconstructed from $\dstat'$ and the values $m_1,\ldots,m_L$. It follows that 
\[
|\{\text{Degree statistics }\dstat: \hat{\dstat}=\dstat'\}| \le 
\Big|\Big\{(m_1,\ldots,m_L)\in \N^L: \sum_{0 < c < L} cm_c < n\Big\}\Big| < n^L\, .
\]

To control (b), note that if $\dstat$ is the degree statistics of a tree of size $n$, then by the formula \eqref{eq:dstat_count} for the number of trees with given degree statistics, we have 
\begin{align*}
\frac{|\tset_{\dstat}|}{|\tset_{\hat{\dstat}}|}
& = 
\prod_{c \ge 0} \frac{\hat{\dstat}(c)!}{\dstat(c)!}\\
& = 
\frac{(\dstat(0)+\sum_{c=1}^L (M-c)m(c))!}{\dstat(0)!}
\cdot
\frac{(\dstat(M)+\sum_{c=1}^L cm(c))!}{\dstat(M)!}
\cdot
\prod_{c=1}^L \frac{(\dstat(c)-Mm(c))!}{\dstat(c)!}\, .
\end{align*}
Since $M > 2L > 2$, we have $\dstat(0)> \dstat(M)$ and $\sum_{c=1}^L (M-c)m(c) \ge \sum_{c=1}^L cm(c)$. Thus 
$(\dstat(0)+\sum_{c=1}^L (M-c)m(c))! > (\dstat(M)+\sum_{c=1}^L cm(c))!$ and so 
\[
\frac{(\dstat(0)+\sum_{c=1}^L (M-c)m(c))!}{\dstat(0)!}
\cdot
\frac{(\dstat(M)+\sum_{c=1}^L cm(c))!}{\dstat(M)!}
\le \frac{(\dstat(0)+\sum_{c=1}^L Mm(c))!}{\dstat(0)!}.
\]
It follows that 
\[
\frac{|\tset_{\dstat}|}{|\tset_{\hat{\dstat}}|}
\le 
\frac{(\dstat(0)+\sum_{c=1}^L Mm(c))!}{\dstat(0)!}
\cdot 
\prod_{c=1}^L \frac{(\dstat(c)-Mm(c))!}{\dstat(c)!}\, .
\]
Since $\dstat(c)-Mm(c) \le M-1$ and 
$Mm(c)\le \dstat(c)$ for all $c \in \N$, this yields the bound 
\begin{equation}\label{eq:set_size_relation}
\frac{|\tset_{\dstat}|}{|\tset_{\hat{\dstat}}|}
\le 
((M-1)!)^L \frac{(\sum_{c=0}^L \dstat(c))!}{\prod_{c=0}^L\dstat(c)!} 
 \le ((M-1)!)^L (L+1)^n\, ,
\end{equation}
where in the last inequality we have used the fact that the final fraction is a multinomial coefficient and that $\sum_{c=0}^L \dstat(c) \le \sum_{c \ge 0} \dstat(c)=n$. 
To conclude, write $N_n$ (resp.\ $\hat{N}_n$) for the set of degree statistics $\dstat$ with $\sum_{c \ge 0}\dstat(c)=n=|\dstat|_1+1$ and such that 
$\sum_{1 \le c \le L} \dstat(c) \ge 2\delta n$ (resp.\ such that $\sum_{1 \le c \le L} \dstat(c) \le \delta n$), and note that if $\dstat \in N_n$ then $\hat{\dstat} \in \hat{N}_n$ provided that $n \ge (M-1)L/\delta$. 
We have 
\begin{align*}
\p{\sum_{1 \le c \le L} \dstat_{\sgt_n}(c) \ge 2\delta n}
& = \sum_{\dstat \in N_n} \p{\dstat_{\sgt_n}=\dstat}\\
& = 
\sum_{\dstat \in N_n}
\sum_{t \in \tset_{\dstat}}
\p{\sgt_n=t}\\
& =
\sum_{\dstat \in N_n}
|\tset_\dstat| \frac{\wts(\dstat)}{Z_n(\wts)} \\
& \le
\sum_{\dstat \in N_n}
((M-1)!)^L (L+1)^n |\tset_{\hat{\dstat}}| 
\frac{\wts(\hat{\dstat})}{Z_n(\wts)}\frac{1}{K^{\delta n}}\,,
\end{align*}
where we have used \eqref{eq:weight_gain_is_good} and \eqref{eq:set_size_relation} for the final bound. 
For each $\dstat' \in \hat{N}_n$, 
there are at most $n^L$ sequences $\dstat \in N_n$ with $\hat{\dstat}=\dstat'$, so for $n \ge (M-1)L/\delta$ the above bound yields 
\begin{align*}
\p{\sum_{1 \le c \le L} \dstat_{\sgt_n}(c) \ge 2\delta n}
& \le n^L((M-1)!)^L ((L+1)/K^\delta)^n \sum_{\dstat' \in \hat{N}_n} 
|\tset_{\dstat'}| 
\frac{\wts(\dstat')}{Z_n(\wts)} \\
& = (n(M-1)!)^L ((L+1)/K^\delta)^n \
\p{\sum_{1 \le c \le L} \dstat_{\sgt_n}(c) \le \delta n}\, .
\end{align*}
Since $K^\delta > 2(L+1)$, the term $(n(M-1)!)^L ((L+1)/K^\delta)^n$ tends to zero as $n \to \infty$, which establishes \eqref{eq:toprove} and completes the proof. 
\end{proof}
The next corollary is the key takeaway from Theorems~\ref{thm:dstats_simply_generated} and~\ref{thm:nu_zero}, for the purposes of this work.
\begin{cor}\label{cor:dstats_simply_generated}
In the setting of Theorems~\ref{thm:dstats_simply_generated} and~\ref{thm:nu_zero}, if $\nu < 1$, or if $\nu=1$ and $\sigma^2=\infty$, then for any $C > 0$, with very high probability 
\[
|\dstat_{\sgt_n}|_2^2 \ge C|\dstat_{\sgt_n}|_1. 
\]
\end{cor}
In the same way that Theorems~\ref{thm:main2} and~\ref{thm:main3} follow from Propositions~\ref{prop:dstats_infinite_variance} and~\ref{prop:dstats_zero_radius}, Corollary~\ref{cor:dstats_simply_generated} implies that if $\nu < 1$, or if $\nu=1$ and $\sigma^2=\infty$, then 
$\hyt(\sgt_n)/(n^{1/2}\log^3 n) \to 0$ and $\wid(\sgt_n)/n^{1/2} \to \infty$ in probability. This proves the second and fourth points of the above conjecture and nearly proves the first and third points, up to the polylogarithmic factors. 
\begin{proof}[Proof of Corollary~\ref{cor:dstats_simply_generated}]
We begin by noting that if the support of $\wts$ is finite then $\rho=\infty$ and thus $\nu = \Psi(\infty) = \max(k:w_k>0) > 1$.

We now argue in three cases. First, suppose that $0<\nu<1$. In this case the support of $\wts$ is infinite, so for any fixed $K\in \N$, $\sum_{k\leq K}k\pi(k) < \sum_{k\geq 0} k\pi(k) = \nu$. 
It follows by Theorem~\ref{thm:dstats_simply_generated} that with very high probability $\sum_{k\leq K} k\dstat_{\sgt_n}(k) \leq \nu n-1$. Since 
\[
\sum_{k> K} k\dstat_{\sgt_n}(k)  
= 
|\dstat_{\sgt_n}|_1 -
\sum_{k\leq K} k\dstat_{\sgt_n}(k)
= n-1-
\sum_{k\leq K} k\dstat_{\sgt_n}(k)\, ,
\]
this implies that with very high probability 
    \begin{equation*}
        \sum_{k> K} k\dstat_{\sgt_n}(k)  
        \geq n(1-\nu).
    \end{equation*}
    Now let $C>0$ arbitrary and fix $K\in \N$ such that $K(1-\nu) >C$. Then 
    since 
    $|\dstat_{\sgt_n}|_2^2 \geq \sum_{k> K} k^2\dstat_{\sgt_n}(k) > K  \sum_{k> K} k\dstat_{\sgt_n}(k)$, we have
    \begin{align*}
        \p{|\dstat_{\sgt_n}|_2^2 < C|\dstat_{\sgt_n}|_1} &\leq \p{ K  \sum_{k> K} k\dstat_{\sgt_n}(k)  < Cn} \\
        &  = \p{ K  \sum_{k> K} k\dstat_{\sgt_n}(k)  < Cn,\, \sum_{k> K} k\dstat_{\sgt_n}(k) <  n(1-\nu)}  = \oexp(1)\, .
    \end{align*}
    This establishes the corollary in the case that $\nu \in (0,1)$.
    
    Next, suppose $\nu = 1$ and $\sigma^2 = \infty$. Let $C>0$ arbitrary. Then since we have $\lim_{K \to \infty} \sum_{k=0}^K k^2 \pi(k) = \infty$, there exists $K \in \N$ such that $\sum_{k=0}^K k^2 \pi(k) \geq 2C$. Noting that $|\dstat_{\sgt_n}|_1 = n-1$, we can write 
\begin{align*}
    \p{|\dstat_{\sgt_n}|_2^2 < C|\dstat_{\sgt_n}|_1} &\leq \p{|\dstat_{\sgt_n}|_2^2 < Cn} \\ 
    &\leq \p{ \sum_{k=0}^K k^2 \frac{\dstat_{\sgt_n} (k)}{n}  < C} \, .
\end{align*}
Then, taking $\eps < 6C/{(K(K+1)(2K+1)})$ and applying Theorem~\ref{thm:dstats_simply_generated}, we have 
\begin{equation*}
    \p{|\dstat_{\sgt_n}|_2^2 < C|\dstat_{\sgt_n}|_1} \leq 
    \p{\sum_{k=0}^K k^2 \frac{\dstat_{\sgt_n} (k)}{n} < C , \bigcap_{k=0}^K \left\{\left|\frac{\dstat_{\sgt_n}(k)}{n} - \pi(k)\right| \leq \eps\right\} } + \oexp(1) \, . 
\end{equation*}
However, if $|\tfrac{\dstat_{\sgt_n}(k)}{n} - \pi(k)| \leq \eps$ for all $0 \le k \le K$ then 
\begin{equation*}
    \sum_{k=0}^K k^2 \frac{\dstat_{\sgt_n}(k)}{n} \geq \sum_{k=0}^K k^2 (\pi(k) - \eps) \geq 2C - \eps \sum_{k=0}^K k^2 > C \, ,
\end{equation*}
so the intersection of events in the probability on the right-hand side is empty and thus this probability is zero. Therefore $\p{|\dstat_{\sgt_n}|_2^2 < C|\dstat_{\sgt_n}|_1}=\oexp(1)$, as required. 

Finally, suppose that $\nu=0$, and fix $C \in \N$.  By Theorem~\ref{thm:nu_zero}, we have 
\[
\p{\sum_{k=1}^C \dstat_{\sgt_{n}}(k) > \frac{n}{2C^2}}
\le 
\p{\frac{\dstat_{\sgt_{n}}(0)}{n} < 1-\frac{1}{2C^2}} = \oexp(1). 
\]
However, if $\sum_{k=1}^C \dstat_{\sgt_{n}}(k) < n/(2C^{2})$ then $\sum_{k=1}^C k\dstat_{\sgt_{n}}(k) < n/(2C)$, so 
\[
\sum_{k > C}k\dstat_{\sgt_{n}}(k)
=
n-1 -\sum_{k \le C} k\dstat_{\sgt_{n}}(k)
\ge 
n-1-\frac{n}{2C}\, ,
\]
and thus 
\[
\sum_{k \ge 1}k^2\dstat_{\sgt_{n}}(k)
\ge
(C+1)\sum_{k > C}k\dstat_{\sgt_{n}}(k)
\ge (C+1)(n-1)-\frac{n(C+1)}{2C} > C(n-1)=C|\dstat_{\sgt_n}|_1\, ,
\]
the last inequality holding for $n$ large provided $C >1$. This shows that $|\dstat_{\sgt_n}|_2^2>C|\dstat_{\sgt_n}|_1$ with very high probability, and completes the proof.
\end{proof}
\begin{proof}[Proofs of Propositions~
\ref{prop:dstats_infinite_variance} and~\ref{prop:dstats_zero_radius}]
Define a weight sequence $\wts$ by $\wts_k = \mu(k)$. Then $\sgt_n=\sgt_n(\wts)$ is distributed as a Bienaym\'e tree with offspring distribution $\mu$, conditioned to have size $n$. 

Now suppose that $\mu$ satisfies the assumptions of either Proposition~\ref{prop:dstats_infinite_variance} or Proposition~\ref{prop:dstats_zero_radius}. 
Then either $|\mu|_2^2=\infty$ or $\sum_{k \ge 0} e^{tk}\mu(k) =\infty$ for all $t > 0$. In either case, $\wts$ has radius of convergence $\rho$ equal to $1$. Thus $\nu = \Psi(\rho) = \sum_{k \ge 0} kw_k = |\mu|_1$, and 
either $\nu < 1$ or else $\nu=1$ and $\sigma^2 = \rho\Psi'(\rho)=\Psi'(1) = \sum_{k \ge 0} k^2 w_k-(\sum_{k \ge 0} kw_k)^2 = |\mu|_2^2-|\mu|_1=\infty$, whence the (common) conclusion of the propositions follows from Corollary~\ref{cor:dstats_simply_generated}. 
\end{proof}
\begin{proof}[Proof of Proposition~\ref{prop:dstats_finite_variance}]
Again define a weight sequence $\wts$ by $\wts_k = \mu(k)$. Then $\sgt_n=\sgt_n(\wts)$ is distributed as a Bienaym\'e tree with offspring distribution $\mu$, conditioned to have size $n$. 

Fix $\eps > 0$ and let $E_n(k)$ be the event that $\dstat_{\sgt_n}(k) \ge n(\mu(k)-\eps/2^k)$. 
By Theorem~\ref{thm:dstats_simply_generated}, $E_n(k)$ happens with very high probability. Now fix $K\ge 2$ large enough that $\mu(K,\infty)< \eps/2$. If $E_n(k)$ occurs for each $2 \le k \le K$ then 
\[
|\dstat_{\sgt_n}|_2^2 \ge \dstat_{\sgt_n}(1) + \sum_{k=2}^K k^2 n (\mu(k)-\eps/2^k) \ge \dstat_{\sgt_n}(1) + 4 n \Big( \sum_{k=2}^K \mu(k) - \sum_{k=2}^K \eps / 2^k \Big) \, , 
\]
in which case, since $|\dstat_{\sgt_n}|_1 = n-1$, we have
\[ 
    |\dstat_{\sgt_n}|_2^2 - \dstat_{\sgt_n}(1) \ge 4 n \Big(1 - \mu(0) - \mu(1) - \mu(K, \infty) - \sum_{k=2}^K \frac{\eps}{2^k} \Big) > |\dstat_{\sgt_n}|_1 \cdot 4 (1 - \mu(0) - \mu(1) - \eps)\,.
\]
Since $\bigcap_{k=2}^K E_n(k)$ occurs with very high probability, the result follows.
\end{proof}

\addtocontents{toc}{\SkipTocEntry} 
\section*{\bf Acknowledgements}
During this work LAB was partially supported by NSERC Discovery Grant 643473. The authors also thank Serte Donderwinkel for useful discussions, and for pointing out a gap in one of the proofs in an early version of the paper (as well as how to fix it); and two referees, whose careful reading substantially improved the paper.


\small 

\bibliographystyle{plainnat}
\bibliography{2021-05_pth_arxiv}

%
%
                 
\appendix

\end{document}